\documentclass{amsart}

\usepackage{amsmath,amssymb}
\usepackage{mathrsfs} 
\usepackage{enumerate}
\usepackage{appendix}

\usepackage{pgf,tikz}
\usetikzlibrary{arrows}
\usepackage[all]{xy}

\usepackage{tikz-cd}
\tikzcdset{arrow style=tikz,diagrams={>=stealth}}

\RequirePackage{hyperref}
\hypersetup{pdfpagemode={UseOutlines},
bookmarksopen=true,
bookmarksopenlevel=0,
hypertexnames=false,
colorlinks=true, 
citecolor=teal, 
linkcolor=blue, 
urlcolor=magenta, 
pdfstartview={FitV},
unicode,
breaklinks=true,
}

\newcommand \R {\mathbf R}
\newcommand \Z {\mathbf Z}
\newcommand \Q {\mathbf Q}
\newcommand \N {\mathbf N}
\renewcommand \C {\mathbf C}

\newtheorem{theorem}{Theorem}[section]
\newtheorem{lemma}[theorem]{Lemma}
\newtheorem{proposition}[theorem]{Proposition}

\newtheorem*{theorem*}{Theorem}
\newtheorem*{ques*}{Question}
\newtheorem*{prop*}{Proposition}

\theoremstyle{definition}
\newtheorem{definition}[theorem]{Definition}
\newtheorem{example}[theorem]{Example}

\theoremstyle{remark}
\newtheorem{remark}[theorem]{Remark}

\numberwithin{equation}{section}

\begin{document}

\title[Sublinear quasiconformality and large-scale geometry]{Sublinear quasiconformality and the large-scale geometry of Heintze groups}


\author{Gabriel Pallier}
\address{}
\address{ 
Laboratoire de Math{\'e}matiques d'Orsay, Univ.\ Paris-Sud, CNRS, Universit{\'e} Paris-Saclay, 91405 Orsay, France.}
\email{gabriel.pallier@math.u-psud.fr}
\email{gabriel.pallier@dm.unipi.it}


\date{\today}
\thanks{G.P. was partially supported by the Agence nationale de la recherche, ANR-15-CE40-0018 SRGI and by European Research Council
 (ERC Starting Grant 713998 GeoMeG `\emph{Geometry of Metric Groups}'}

\dedicatory{}

\begin{abstract}
    This article analyzes sublinearly quasisymmetric homeomorphisms (generalized quasisymmetric mappings), and draws applications to the sublinear large-scale geometry of negatively curved groups and spaces. 
    It is proven that those homeomorphisms lack analytical properties but preserve a conformal dimension and appropriate function spaces, distinguishing certain (nonsymmetric) Riemannian negatively curved homogeneous spaces, 
    and Fuchsian buildings, up to sublinearly biLipschitz equivalence (generalized quasiisometry).
\end{abstract}

\maketitle

An embedding $f$ between metric spaces is quasisymmetric if there is an increasing homeomorphism $\eta:[0, +\infty) \to [0, + \infty)$ such that for any $x,y,z$ in the source space and positive real $t$,
\begin{equation}
    \label{eq:definition-quasisym-homeo}
d(x,y) \leqslant t d(x,z) \implies d(f(x), f(y)) \leqslant \eta(t) d(f(x), f(z)).
\end{equation}
The properties of sufficiently well-behaved compact metric spaces that are invariant under quasisymmetric homeomorphisms are known to be counterparts of the coarse (or quasiisometrically invariant) properties of proper geodesic Gromov-hyperbolic spaces, the two categories being related by the Gromov boundary and hyperbolic cone functors (\cite{BonkSchrammEmbeddings}, \cite[2.5]{RoeLectures}). Instances are the conformal dimension \cite{PansuDimConf} and the $\ell_p$ or $L_p$ cohomology  \cite{BourdonPajotCohomologieLp}. 

This paper is part of our aim to transpose this equivalence by replacing quasiisometries with sublinearly biLipschitz equivalences, which originated from the work of Cornulier on the asymptotic cones of connected Lie groups\footnote{Beware that we use the terminology of \cite{cornulier2017sublinear}.} \cite{CornulierDimCone}. Here the sublinear feature is described by an asymptotic class $O(u)$, where $u$ is a strictly sublinear nondecreasing positive function on the half line such that $\limsup_r u(2r)/u(r)<+\infty$, e.g. $u(r) =\log r$ (we call such a function admissible).

In previous work the Gromov-boundary behavior of sublinearly biLipschitz equivalences between Gromov-hyperbolic spaces was characterized \cite[Theorem 1]{pallier2018large}. It differs from that of quasisymmetric homeomorphisms sublinearly in a certain sense; we shall indicate how in \ref{subsec:sublinear-quasiconformality}. 
The purpose of the present paper is to push further the analysis of those boundary mappings and identify the structure preserved on the boundary. 
A numerical invariant is derived. It is denoted by $\operatorname{Cdim}_{O(u)}$; Pansu's conformal dimension introduced in \cite[3]{PansuDimConf} and usually denoted $\operatorname{Cdim}$ corresponds to $\operatorname{Cdim}_{O(1)}$. 
We compute this invariant and prove that it equals $\operatorname{Cdim}$ on the examples originally studied by Pansu and Bourdon (that we recall below).
Certain function spaces of locally bounded $p$-variation, that are carried by sublinearly quasisymmetric mappings up to shifts in parameters, are also constructed. These functions are invariant along foliations in the boundary; the dependence of this invariance with respect to $p$ provides further invariants. This latter approach is inspired from Bourdon \cite[p.248]{BourdonCaract}, Bourdon-Kleiner \cite[Section 10]{BourdonKleinerCLPi} and Carrasco Piaggio \cite[p.465]{CarrascoOrliczHeintze} (with different functional spaces). Functions of bounded $p$-variation were also used in Xie's work \cite{XieLargeScale} on a problem close to ours that we will mention below.

A purely real Heintze group is a simply connected solvable group which splits as an extension of $\R$ by its nilradical $N$, associated to $\rho:\R\to\operatorname{Aut}(\operatorname{Lie}(N))$ with positive real roots.
From such a group $S$ one can make another one, denoted $S_{\infty}$, by forgetting the unipotent part of $\rho$. Since the nilradical of $S$ is uniformly exponentially distorted, following Cornulier one can prove that this does not alter the logarithmic sublinear large-scale structure (see \cite[Th 1.2]{CornulierCones11} recalled here in \ref{subsubsec:diagonalizable-type}).
We prove a partial converse.

\begin{theorem*}
\label{thm:roots}
Let $S$ and $S'$ be purely real Heintze groups with abelian nilradicals. Let $u$ be any sublinear, admissible function. 
If $S$ and $S'$ are $O(u)$-sublinearly biLipschitz equivalent then $S_{\infty}$ and $S'_\infty$ are isomorphic.
\end{theorem*}

This answers positively to Cornulier {\cite[1.16(1)]{cornulier2017sublinear}} who raised the question for $\dim S =3$.
For comparison, it is known that two purely real Heintze groups with abelian nilradicals are quasiisometric if and only if they are isomorphic by the work of Xie \cite{XieLargeScale} (also obtained by Carrasco Piaggio \cite[1.10]{CarrascoOrliczHeintze}).
In the vein of conjecture \cite[6C2]{CornulierQIHLC}, we ask:

\begin{ques*}
Let $S$ and $S'$ be purely real Heintze group. Assume that $S$ and $S'$ are sublinearly biLipschitz equivalent. Are $S_{\infty}$ and $S_{\infty}'$ isomorphic?
\end{ques*}

A positive answer would imply the previous theorem as well as  \cite[Theorem 2]{pallier2018large}. 
The classification problem can be motivated beyond Lie groups by the fact that the purely real Heintze groups are known to parametrize other objects:
\begin{itemize}
    \item 
    The commability\footnote{Namely, to such a group $G$ one can associate the purely real core of the unique focal-universal group commable to $G$. Commability is a variant of weak commensurability adapted to the locally compact setting, see \cite{Cornulier_Focal}. 
    For the definition of a hyperbolic locally compact group we refer the reader to \cite{CCMT}.} classes of compactly generated locally compact groups that are hyperbolic with a topological sphere at infinity \cite[5.16]{Cornulier_Focal}.
    \item
    Together with orbits of scalar products, the connected Riemannian negatively curved homogeneous spaces \cite{Heintze} \cite[Corollary 5.3]{GordonWilson}.
\end{itemize}
Unlike Heintze groups, hyperbolic buildings become rare in large dimension \cite{GaboPaul}. 
The two-dimensional case displays a vast subfamily with local finiteness properties, that of Fuchsian buildings, for which the dimension at infinity $\operatorname{Cdim} \partial_\infty$ is known: it was computed by Bourdon in 1997 \cite{BourdonFuchsI} for some of them and 2000 in full generality \cite{BourdonFuchsII}. We check that $\operatorname{Cdim}_{O(u)}\partial_\infty$ equals the former in this case, distinguishing pairs of Fuchsian buildings up to sublinear biLipschitz equivalence. Here is the statement for the Bourdon buildings.

\begin{prop*}[Strengthening of {\cite[Th{\'e}or{\`e}me 1.1]{BourdonFuchsI}}]
Let $p,q \in \Z$ with $p \geqslant 5$ and $q \geqslant 2$.
Let $I_{pq}$ be a Bourdon building (that is, a right-angled Fuchsian building with constant thickness $q$). For all strictly sublinear admissible $u$,
\begin{equation}
	\label{eq:confdim-bourdon-buildings}
    \operatorname{Cdim}_{O(u)} \partial_\infty I_{pq} = \operatorname{Cdim}_{O(1)} \partial_\infty I_{pq} = 1 + \frac{\log(q-1)}{\operatorname{argch} ((p-2)/2)}.
\end{equation}
\end{prop*}

\subsubsection*{Conventions, notation}
    Through all the paper, $u : \R_{\geqslant 0} \to \R_{\geqslant 1}$ is a nondecreasing, strictly sublinear, doubling function, i.e. $u(r) \ll r$ as $r\to + \infty$ and $\sup_r u(2r)/u(r) < +\infty$. Examples are: $u(r) = \sup(1,r^\gamma)$ with $0 \leqslant \gamma < 1$ and $u(r) = \sup (1, \log (r)) $.
\label{chap:sublinconf}

\subsubsection*{Acknowledgement}
This work is part of the author's PhD thesis.
The author thanks his advisor Pierre Pansu for his long-time support and patience, Yves Cornulier for raising questions and pointing out \cite{DranishnikovSmith}, John Mackay for his interest and providing references, Pierre Boutaud, Arnaud Durand, and Anthony Genevois for useful discussions, Peter Ha{\"i}ssinsky for numerous remarks and corrections on the text.

\section{Sublinear quasiconformality}

\subsection{{$O(u)$}-quasisymmetric structures}

The notion of a quasisymmetric structure is a reformulation of that of a space with a quasidistance, where the emphasis is made on balls, their inclusion relations and relative sizes, rather than on a given quasidistance function. 
Related notions are: $b$-metric topological spaces \cite[IV.1]{MarDS}, Margulis structures \cite[p.62]{GroPanRigid}.

\label{subsec:definitions-sublin-conf}

\subsubsection{Definition}

\begin{definition}[Compare\footnote{In \cite[1.1 and 2.7]{PansuDimConf} they are called ``bonnes structures quasiconformes''. The ``bonne'' axiom is a disguised form of the quasi-triangle inequality, here (SC2).} {\cite[1.1 and 2.7]{PansuDimConf}} for $u=1$]
\label{dfn:sublin-conf-struct}
Let $Z$ be a set. 
A $O(u)$-quasisymmetric structure on $Z$ is a set $\beta$ of abstract balls\footnote{This formalism is here to avoid referring directly to centers and radii, which are preferable to diameters, but may not be uniquely defined. 
The notion of a constituent (see \cite[Definition 2]{EdgarPacking}) circumvents the problem of radii, but it still makes use of centers.} together with a realization map $\beta \to \mathcal{P}(Z) \setminus \lbrace \emptyset \rbrace$, $b \mapsto \widehat{b}$, a map ${\delta} : \beta \to \Z$ and a shift map $\mathbf Z_{\geqslant 0} \times \beta \to \beta$, $(k,b) \mapsto k.b$ such that
\begin{enumerate}[(SC1)]
\setcounter{enumi}{-1}
    \item 
    The shift is an action and $\delta$ is equivariant with respect to the shift:
    $\forall k,k' \in \mathbf Z_{\geqslant 0}$, $k'.k.b = (k'+k).b$ and
    $\forall k \in \Z, \forall b \in \beta,\, {\delta}(k.b) = {\delta}(b) -k$.
    \item
    $\forall k \in \Z_{\geqslant 0}, \, \forall b,b' \in \beta$, 
    \begin{enumerate}[(i)]
        \item 
        $\widehat{k.b} \supseteq \widehat{b}$
        \item
        if
        $\widehat{b} \subseteq \widehat{b'}$ then $\widehat{k.b} \subseteq \widehat{k.b'}$
        \item
        if $\delta(b)<\delta(b')$ then $\widehat{b} \nsubseteq \widehat{b'}$.
    \end{enumerate}
    \item
    There exists $n_0 \in \Z_{\geqslant 0}$ and a function $q : \Z_{\geqslant n_0} \to \Z_{\geqslant 0}$, $q = O(u)$ and  such that
    \[ \forall b, b' \in \beta,\, \left( n_0 \leqslant {\delta}(b) \leqslant {\delta}(b'), \widehat{b} \cap \widehat{b'} \neq \emptyset \right) \implies \widehat {q({\delta}(b)).b} \supset \widehat{b'}.  \]
    \item
    $\forall x \in Z, \forall y \in Z \setminus \lbrace x \rbrace, \forall n \in \Z, \exists b \in \beta : {\delta}(b) \geqslant n, x \in \widehat b, y \notin \widehat b$.
\end{enumerate}
\end{definition}

\begin{example}[Space with a quasidistance]
\label{exm:space-with-a-quasidist}
Recall that a quasidistance on a set $\mathcal{Z}$ is a kernel $\varrho : \mathcal{Z} \times \mathcal{Z} \to \R$ with the axioms of a distance, the triangle inequality being replaced by
\begin{equation}
    \tag{$\triangle_K$}
    \label{ineq:triangle-K}
    \forall(x,y,z) \in \mathcal Z^3,\, \varrho(x,z) \leqslant K\left( \varrho(x,y) \vee \varrho(y,z) \right)
\end{equation}
where $K \in \R_{\geqslant 1}$ is a constant and $\vee$ denotes the binary function ``max''. Given a dense\footnote{A quasidistance induces a topology, see \cite[1.99]{PajotRuss}.} subspace (to be thought of as a set of centers) $X \subseteq \mathcal{Z}$, a quasidistance gives to $\mathcal{Z}$ a $O(1)$-quasisymmetric structure in which $\beta = X \times \Z$ and for $b = (x,n)$ in $\beta$ and $k \in \Z$, ${\delta}(b) =n$, $k.b=(x,n-k)$ and 
$\widehat{b} = \left\{ z \in Z : \varrho(x,z) \leqslant e^{-n} \right\}$.
\eqref{ineq:triangle-K} is responsible for (SC2) with $q = K^2$, the separation axiom for (SC3).
\end{example}

\begin{example}
    \label{exm:reals}
    $Z = \R$ and $\beta$ is $\R \times \Z$; for $b=(x,n)$, ${\delta}(b)=-n$. 
    For all $b=(s,n)$ in $\beta$, $\widehat{b} = s+e^{-n}[0,1]$ (One can replace $[0,1]$ by any bounded closed interval). One can take $q=3/2$ in (SC2). The shift is such that $\widehat{k.b} = e^k \widehat{b}$.
\end{example}

It turns out that once $Z$ is endowed with a sublinear quasisymmetric structure, $Z$ is also equipped with the structure (and especially the topology) of a uniform space, that is a weakening of a metric structure in a sense that we recall in the statement below.

\begin{proposition}
\label{prop:uniform-structure}
Let $(Z,\beta,q,\delta)$ be a $O(u)$-quasisymmetric structure.
For all $n \in \Z$, define 
\begin{equation*}
    \label{eq:uniform-structure}
    E_n = \bigcup_{b \in \beta : {\delta}(b) \geqslant n} \widehat{b} \times \widehat{b}. 
\end{equation*}
Then $E_n$ forms a fundamental system of entourages, endowing $Z$ with a uniform structure, i.e. (denoting $\Delta$ the diagonal in $Z \times Z$):
\begin{itemize}
\item[(U'\textsubscript{I})]
$\cap_n E_n = \Delta$
\item[(U'\textsubscript{II})]
for every $n,m$ there is $p$ such that $E_p \subset E_n \cap E_m$
\item[(U'\textsubscript{III})]
for every $n$ there is $m$ such that
\begin{equation}
    \forall x,y,z \in Z, \, \lbrace (x,y) \rbrace \cup \lbrace (y,z) \rbrace \subset E_m \implies (x,z) \in E_n.
    \tag{$E_m^2 \subset E_n$}
    \label{eq:goal-for-prop-uniform}
\end{equation}
\end{itemize}
\end{proposition}

(See \cite[p.8]{Weil1937espaces} for the original set of axioms (U') and equivalent ones; we use a slight simplification of (U'\textsubscript{III}) in view of the fact that the $E_n$ are stable under $(x,y) \mapsto (y,x)$.)

\begin{proof}
(U'\textsubscript{I}) follows from (SC3) and (U'\textsubscript{II}) from the definition, setting $p = n \vee m$.
To check  (U'\textsubscript{III}), letting $n \in \mathbf Z$  one needs to find $m \in \Z$ such that \eqref{eq:goal-for-prop-uniform} holds.  This can be rephrased as follows:
for any pair of distinct $x,y \in Z$, set
\begin{equation*}
    \varrho(x,y) = \exp \left( - \inf \left\{ {\delta}(b) : b \in \beta , \lbrace x \rbrace \cup \lbrace y \rbrace \subset b \right\} \right) 
\end{equation*}
and $\varrho(x,x) = 0$. 
Especially $(x,y) \in E_n \iff - \log \varrho(x,y) \geqslant n$.
Then for all $x,y,z \in Z^3$
\begin{equation}
    \tag{$\triangle_{O(u)}$}
    \label{ineq:triangle-u}
    \varrho(x,z) \leqslant e^{ v \left( \inf\left\{ - \log \varrho(x,y), - \log \varrho(y,z) \right\} \right) } \left[ \varrho(x,y) \vee \varrho(y,z) \right]
\end{equation}
where $v =O(u)$ (one may take $v(n) = q(n)$ at least for $n \geqslant n_0$). Set $m_0 = 2n \vee 2 \sup \left\{ m' \in \Z_{\geqslant 0} : v(m') \geqslant \frac{m'}{2} \right\}$. Then for every $m \geqslant m_0$, $m - v(m) \geqslant m - m/2 = m/2 \geqslant n$. \eqref{eq:goal-for-prop-uniform} is achieved.
\end{proof}

If $(Z, \beta)$ is as above, the topology is defined on $Z$ by declaring $\Omega \subset Z$ open if for every $x \in \Omega$ there is $n\in \mathbf Z$ such that for all $y\in Z$, $(x,y) \in E_n$ implies $y \in \Omega$.

\begin{remark}
\label{rem:topology}
An open subspace $\Omega$ of a $O(u)$-quasisymmetric structure $(Z,\beta)$ inherits a $O(u)$-quasisymmetric structure $(\Omega,\beta_{\mid \Omega})$ where \[\beta_{\mid \Omega}=\left\{ b \in \beta : \forall k \in \mathbf{Z}_{\geqslant 0}, \widehat{k.b} \cap \Omega \neq \emptyset \right\},\]
the shift is restricted to $\beta_{\mid \Omega}$, and the realization is $\widehat{b}_{\mid \Omega}=\widehat{b}\cap\Omega$.
\end{remark}

\subsubsection[Hyperbolic cones]{Hyperbolic cones and sublinear large-scale geometry}
\label{subsec:Gromov-hyp}

The boundary of a Gromov-hyperbolic space has a Margulis structure, see e.g.\ \cite{GroPanRigid}; further, the boundary construction can be reversed as suggested by M.Gromov \cite[1.8.A(b)]{Gromov87HypGrp} and elaborated by M. Bonk and O. Schramm (\cite[\S\ 7]{BonkSchrammEmbeddings}, see also \cite{PaulinGrpHypBord}), so that in the current formalism any $O(1)$-quasisymmetric structure occurs at the boundary of a Gromov-hyperbolic space\footnote{Namely a certain quotient space of $\beta$, two abstract balls being close if close for ${\delta}$ and if their realizations intersect, compare e.g. \cite[chapter 2]{RoeLectures}. 
Abstract, resp.\ concrete balls are turned into geodesic segments, resp.\ their endpoints.  
The metric hyperbolicity is implied by (SC1) and (SC2).}.
It is a classical fact that quasiisometries between Gromov hyperbolic groups extend to biH{\"o}lder, quasisymmetric homeomorphism between their boundaries, i.e. they do so in a way that preserves the features of the $O(1)$-quasisymmetric structure. 
This paper is rather concerned with sublinearly biLipschitz maps, for which we recall the definition:

\begin{definition}[Cornulier, {\cite{cornulier2017sublinear}}]
\label{dfn:SBE}
Let $(Y,o)$ and $(Y',o')$ be metric spaces. A $O(u)$-sublinearly biLipschitz equivalence (SBE) is a map $f: Y \to Y'$ for which there exists $\lambda \in \R_{\geqslant 1}$ and $v = O(u)$ such that
\begin{enumerate}
    \item 
    $\forall y_1, y_2 \in Y, \frac{1}{\lambda} d(y_1, y_2) - v(\sup \lbrace d(o,y_1), d(o,y_2) \rbrace) \leqslant d(f(y_1), f(y_2))$
    \item 
    $\forall y_1, y_2 \in Y, \lambda d(y_1, y_2) + v(\sup \lbrace d(o,y_1), d(o,y_2) \rbrace) \geqslant d(f(y_1), f(y_2))$
    \item
    $\forall y' \in Y', \exists y \in Y, d(y', f(y)) \leqslant v(d(y,o))$.
\end{enumerate}
\end{definition}

Unlike quasiisometries (which are the $O(u)$-SBE with $u=1$), SBEs are not coarse equivalences in general. 
However they do preserve certain coarse sublinear structures in the sense of Dranishnikov and Smith \cite[2]{DranishnikovSmith}, or large-scale sublinear structures in the sense of Dydak and Hoffland \cite[p.1014]{DydakHoffland}. $O(u)$-quasisymmetric structures are boundary analogs of the former, in a more specific way where $u$ is explicit.
In all our applications $Y$ and $Y'$ will be Gromov-hyperbolic, proper geodesic metric spaces.
Boundary maps of sublinearly biLipschitz equivalences are still homeomorphisms, however a notion more general than quasiconformality needs to be defined.

\subsection{{$O(u)$}-quasisymmetric homeomorphisms}
\label{subsec:sublinear-quasiconformality}

\subsubsection[Definition, comparison with q.s. mappings]{Definition and comparison with quasisymmetric mappings}
\label{def-of-dotplus}

Denote by ${\mathcal{O}^+(u)}$ the semigroup of germs of functions $v$ valued in $\Z_{\geqslant 0}$, defined on large enough integers, such that $v = O(u)$, with the composition law $\dotplus$ defined as 
\begin{equation}
    \label{eq:compostion-law}
    (v_1 \dotplus v_2 )(n) = v_2 (n) + v_1 (n - v_2(n))
\end{equation}
for $n \in \Z$ large enough. 
The reason for this composition law is the requirement that $(\operatorname{Id}-v_1) \circ (\operatorname{Id}-v_2)=\operatorname{Id}-(v_1 \dotplus v_2)$.
$\Z_{\geqslant 0}$ embeds in $\mathcal{O}^+{(u)}$ as the commutative subsemigroup\footnote{The noncommutativity of $\dotplus$ should not be a concern; one can check that for some $C\geqslant 1$, $ C^{-1}(v_1 +v_2)(n)\leqslant (v_1\dotplus v_2)(n) \leqslant C(v_1 +v_2)(n)$ for $n$ large enough.} of constant functions.
$\mathcal{O}^+(u)$ acts on small enough abstract balls: for every $v$ in $\mathcal{O}^+(u)$ there exists $n_0 \in Z$ such that $\Z_{\geqslant n_0}$ lies in the domain of $v$ and if ${\delta}(b) \geqslant n_0$ then $v.b$ is defined as
$v.b = v({\delta}(b)).b$.

\begin{definition}[round sets and rings, compare {\cite[3.4]{TysonQuasiconfQuasisymmm}}]
\label{dfn:round-sets-rings}
    Let $\beta \to \mathcal{P}(Z)$ be a $O(u)$-quasisymmetric structure.
    Given $k \in \mathcal{O}^+(u)$ and $n \in \Z$, a subset $a \in \mathcal{P}(X)$ is a $(k,n)$-round set (or simply a $k$-round set) if there exists $b \in \beta$ such that ${\delta}(b) \geqslant n$ and $\widehat b \subseteq a \subseteq \widehat{k.b}$. 
    A couple of subsets $(a^-,a^+) \in \mathcal{P}(X)^2$ is a $(k,n)$-ring if there exists $b \in \beta$ such that ${\delta}(b) \geqslant n$ and $\widehat b \subseteq a^- \subseteq a^+ \subseteq \widehat{k.b}$.
    Denote by $\mathscr{B}^k_n(\beta)$ resp.\ $\mathscr{R}^k_n(\beta)$ the collection of $(k,n)$-round sets, resp.\ of $(k,n)$-rings, and $\mathscr{B}^k(\beta)$ resp.\ $\mathscr{R}^k(\beta)$ their union over $n \in \mathrm{domain}(k)$.
\end{definition}

\begin{definition}[outer rings]
    \label{dfn:outer-ring}
    Let $\beta \to \mathcal{P}(Z)$ be a $O(u)$-quasisymmetric structure.
    Given $j \in \mathcal{O}^+(u)$, a pair of subsets $(a^-,a^+) \in \mathcal{P}(X)^2$ is a $(j,n)$-outer ring if there exists $n \in \Z$ and $b \in \beta$ such that $a^- \subseteq \widehat{b} \subseteq \widehat{j.b} \subseteq a^+$ and $\delta(b) \geqslant n$. 
    Denote by $\mathscr{O}_{j;n}(\beta)$ the collection of $(j,n)$ outer rings.
\end{definition}

The reader may think of $k$ as a parameter of asphericity\footnote{We borrow the term ``asphericity'' from the survey \cite[p.88]{GroPanRigid}. 
Another choice is ``modulus'' adopted in \cite{pallier2018large} but it would be misleading here since for our purposes in section \ref{sec:conformal-invariants}, moduli are global rather than infinitesimal conformal invariants. Still another term is ``eccentricity''. We prefer not to define asphericity for subsets since we would face the same issues as with radii and centers.} (akin to $\log t$ in \eqref{eq:definition-quasisym-homeo}) that depends on the scale. Whereas quasisymmetric mappings preserve bounded asphericities, ${O(u)}$-quasisymmetric homeomorphisms will be asked to preserve asphericities within the $O(u)$ class. 
We define them in two steps.

\begin{definition}[Equivalent $O(u)$-quasisymmetric structures]
\label{dfn:equivalence}
Let $\beta$ and $\beta'$ be two $O(u)$-quasisymmetric structures on a set $Z$. $\beta'$ is finer than $\beta$ if there exists $\lambda \in \R_{> 0}$ and $n_0 \in \Z$ such that 
\begin{equation}
    \forall k \in \mathcal{O}^+(u),\,  \exists k' \in \mathcal{O}^+(u) : \forall n \in \Z_{\geqslant n_0},\, \mathscr{R}_n^k(\beta) \subseteq \mathscr{R}_{\lfloor \lambda n \rfloor}^{k'}(\beta') 
    \label{eq:finer-conformal-structure}
\end{equation}
\begin{equation}
    \forall j' \in \mathcal{O}^+(u),\,  \exists j \in \mathcal{O}^+(u) : \forall n \in \Z_{\geqslant n_0},\, \mathscr{O}_{j;n}(\beta) \subseteq \mathscr{O}_{j';\lfloor \lambda n \rfloor}(\beta') 
    \label{eq:finer-conformal-structure-outer-rings}
\end{equation}

$\beta$ and $\beta'$ are said equivalent if both finer than each other. 
Up to taking logarithms $k'$ plays with respect to $k$ in \eqref{eq:finer-conformal-structure} the r{\^o}le of $\eta(t)$ with respect to $t$ in \eqref{eq:definition-quasisym-homeo}, so that we will still denote $\eta : \mathcal{O}^+(u) \to \mathcal{O}^+(u)$ a map such that one may take $k' = \eta(k)$ in \eqref{eq:finer-conformal-structure}. 
Similarly, denote $\overline{\eta}:\mathcal{O}^+(u) \to \mathcal{O}^+(u)$ a map such that one may take $j = \overline{\eta}(j')$ in \eqref{eq:finer-conformal-structure-outer-rings}.
$\lambda$ is analogous to a H{\"o}lder exponent comparing snowflake-equivalent metrics.
\end{definition}

\begin{definition}[$O(u)$-quasisymmetric homeomorphism]
\label{dfn:sublin-conf-homeo}
Let $\varphi : Z \to Z'$ be a bijection between two sets endowed with $O(u)$-quasisymmetric structures $\beta$ and $\beta'$. 
One can pull-back $\beta'$ to $Z$ by means of $\varphi$. The map $\varphi$ is a $O(u)$-quasisymmetric homeomorphism if $\beta$ and $\varphi^\ast \beta'$ are $O(u)$-equivalent.
\end{definition}

Two $O(u)$-equivalent structures on $Z$ define the same uniform structure on $Z$ so that ${O(u)}$-conformal homeomorphisms are uniform homeomorphisms. This can be made more quantitative: they are biH{\"o}lder continuous when this makes sense {\cite[4.4]{pallier2018large}}. 
Not every quasi-symmetric homeomorphism is $O(1)$-quasisymmetric, but every power-quasisymmetric homeomorphisms\footnote{A power quasisymmetric embedding is an embedding for which one can take $\eta(t) = \sup\lbrace t^\alpha, t^{1/\alpha} \rbrace$ for some $\alpha \in (0,+\infty)$ in \eqref{eq:definition-quasisym-homeo}; this is not restrictive between uniformly perfect metric spaces (called ``homogeneously dense'' by Tukia and V{\"a}is{\"a}l{\"a}) \cite[3.12]{TukiaVaisalaQSEmbed} \cite[11.3]{HeinonenLectures}.} is.
Note that a consequence of Definition \ref{dfn:sublin-conf-homeo} is that
\begin{equation*}
    \label{eq:weak-finer-conformal}
    \forall k \in \mathcal{O}^+(u),\,  \exists k' \in \mathcal{O}^+(u), \mathscr{B}^k(\beta) \subseteq \mathscr{B}^{k'}(\varphi^\ast \beta')
\end{equation*}
since $k$-round sets may be identified with the $k$-annuli $(a^-, a^+)$ for which there is equality $a^- = a^+$. 
This does not suffice for all our needs, nevertheless it is simpler and we shall use it when possible. 

\begin{remark}
A reformulation of \eqref{eq:finer-conformal-structure} and \eqref{eq:finer-conformal-structure-outer-rings} is
\begin{equation*}
    \forall K \in [1,+\infty), \, \exists K' \in [1,+\infty),\, \mathscr{R}_n^{\lceil Ku \rceil}(\beta) \subseteq \mathscr{R}_{\lfloor \lambda n \rfloor}^{\lfloor K'u \rfloor}(\beta').
\end{equation*}
\begin{equation*}
    \forall J' \in [1,+\infty), \, \exists J \in [1,+\infty),\, \mathscr{O}_n^{\lfloor Ju \rfloor}(\beta) \subseteq \mathscr{O}_{\lfloor \lambda n \rfloor}^{\lceil J'u \rceil}(\beta').
\end{equation*}
\end{remark}

\begin{remark}
The requirement \eqref{eq:finer-conformal-structure-outer-rings} will be needed only when we deal with packings.
\end{remark}

\subsubsection[{$O(u)$}-q.s. homeomorphisms as boundary mappings]{{$O(u)$}-quasisymmetric homeomorphisms as boundary mappings}

{If $Y$ is a proper geodesic Gromov-hyperbolic space, we call visual kernel on the Gromov boundary $Z = \partial_\infty Y$ a function $\rho : Z \times Z \to \R_{\geqslant 0}$ such that $\rho(\xi,\eta) = \exp - (\xi, \eta)_o$ for $\xi, \eta \in Z$, where $(\xi, \eta)_o$ denotes the Gromov product of $\xi$ and $\eta$ seen from $o\in Y$ (this is $\sup \liminf_{i,j} (\xi_i, \eta_j)_o$ for over all sequences $\xi_i \to \xi$, $\eta_j \to \eta$)}.

\begin{theorem}
\label{thm:pallier-SBE2mobius}
Let $Y$ and $Y'$ be Gromov-hyperbolic, geodesic, proper metric spaces with uniformly perfect Gromov boundaries $Z$ and $Z'$.
Let $f : Y \to Y'$ be a $O(u)$-sublinearly biLipschitz equivalence (Definition \ref{dfn:SBE}). 
Let $\beta$ and $\beta'$ be the $O(u)$-quasisymmetric structures on the Gromov boundaries of $Z$ and $Z'$ associated to visual kernels. 
Then $f$ induces a map $\partial_\infty f : Z \to Z'$ between Gromov boundaries is a $O(u)$-quasisymmetric homeomorphism.
\end{theorem}

Since the original statement is not this one, we give details on how to deduce it from \cite{pallier2018large}. 

\begin{figure}
    \begin{tikzpicture}[line cap=round,line join=round,>=angle 90,x=0.5cm,y=0.5cm]
    \clip(-5,-5) rectangle (19,5);
    \fill[fill=blue,fill opacity=0.12] (-4,4) -- (-4,-4) -- (4,-4) -- (4,4) -- cycle;
    \fill[fill=blue,fill opacity=0.15] (-2.936,1.47) -- (0.,-1.47) -- (2.936,-1.47) -- (0.,1.47) -- cycle;
    \path[fill=blue,fill opacity=0.2] (-1.617,0.538) -- (0.543,-0.538) -- (1.617,-0.538) -- (-0.543,0.538) -- cycle;
    \draw [->, line width = 0.2pt] (0,-5) -- (0,5);
    \draw [->, line width = 0.2pt] (-5,0) -- (5,0);
    \draw [line width = 0.1pt, dash pattern =on 3pt off 2pt] (-4,4) --(4,4) -- (4,-4) -- (-4,-4) -- (-4,4) ;
    \draw [line width = 0.1pt, dash pattern =on 3pt off 2pt] (-1.47,1.47) --(1.47,1.47) -- (1.47,-1.47) -- (-1.47,-1.47) -- (-1.47,1.47) ;
    \draw [line width = 0.1pt, dash pattern =on 3pt off 2pt] (-0.54,0.54) --(0.54,0.54) -- (0.54,-0.54) -- (-0.54,-0.54) -- (-0.54,0.54)   ;
    \draw [shift={(12,0)}, line width = 0.1pt, dash pattern =on 3pt off 2pt] (-1.62,1.62) --(1.62,1.62) -- (1.62,-1.62) -- (-1.62,-1.62) -- (-1.62,1.62) ;
    \draw [opacity = 0.5] (-0.54,-0.54) --(10.38,-1.62);
    \draw [opacity = 0.5] (-0.54,0.54) --(10.38,1.62);
    \draw [opacity = 0.0] (0.54,0.54) --(13.62,1.62);
    \draw [opacity = 0.0] (0.54,-0.54) --(13.62,-1.62);
    \begin{scriptsize}
    \draw (0.6,4.3) node {$1$};
    \draw (0.8,1.75) node {$\small{e^{-1}}$};
    \draw (-0.8,0.85) node {$\small{e^{-2}}$};
    \end{scriptsize}
    \draw (-0.2,4) -- (0.2,4);
    \draw (-0.2,1.47) -- (0.2,1.47);
    \draw (-0.2,0.54) -- (0.2,0.54);
    \draw [->, line width = 0.2pt] (12,-5) -- (12,5);
    \draw [->, line width = 0.2pt] (7,0) -- (17,0);
    \fill [scale=1.5, shift={(8,0)},fill=blue,fill opacity=0.25] (-3.24,1.08) -- (-1.08,1.08) -- (3.24,-1.08) -- (1.08,-1.08) -- cycle;
    \draw [>-<, line width = 0.2pt] (6,1.7) -- (15,-3);
    \draw [shift={(-5,3)},>-<, line width = 0.2pt] (11.2,-0.8) -- (12.8,1);
    \draw (8.2,4.5)  node {$\sim s$};
    \draw (16.9,-3)  node {$\sim s \vert \log s \vert$};
    \draw (14.7,-1.3) [color =blue, opacity = 0.7] node {$B(s)$};
    \draw (2.8,-3.3) [color =blue, opacity = 0.6] node {$B(1)$};
    \draw (11.8,1.62) -- (12.2,1.62);
    \draw (12.5,1.9) node{$s$};
\end{tikzpicture}
    \caption[The plane and the tilted plane]{Concentric balls of a quasidistance on $\R^2$ that is invariant under translation and dilation by $\exp(t\alpha')$ with unipotent, non identity $\alpha'$, and coincides with the $\ell^\infty$ distance for pairs of points at distance $1$. For comparison, dashed $\ell^\infty$ spheres of equal radii. Compare Figure \ref{fig:example-illustr-thm-diag}.}
    \label{fig:example-illustr-thm}
\end{figure}
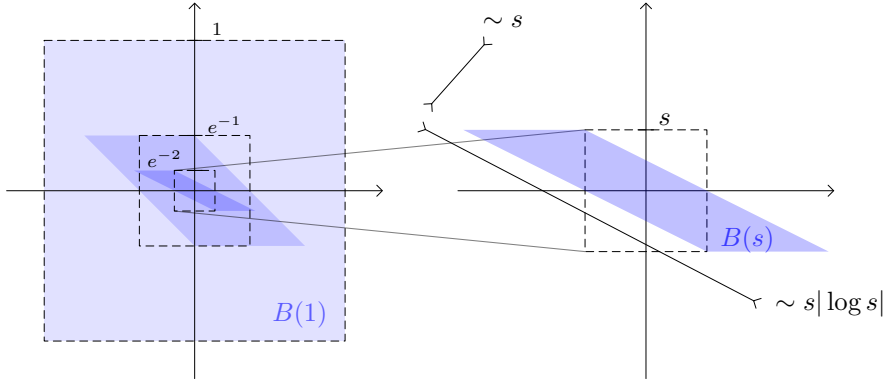

\begin{proof}[How to deduce Theorem {\ref{thm:pallier-SBE2mobius}} from \cite{pallier2018large}]

Fix visual kernels $d$ on $Z$ and $Z'$, start assuming for simplicity that every metric sphere of positive radius in $Z$ and $Z'$ has at least one point, denote $\varphi = \partial_\infty f$; then $\varphi$ and $\varphi^{-1}$ are biH{\"o}lder \cite{cornulier2017sublinear}; up to snowflaking $Z$ or $Z'$ let $\gamma\in(0,1)$ be a H{\"o}lder exponent for both.
By\footnote{Beware that one must translate ``annulus'' into ``ring'' and ``modulus'' into ``asphericity'' to conform to our current terminology.}  \cite[Proposition 4.9]{pallier2018large} sufficiently small rings of inner radius $r$ and asphericity $\log(R/r)$ are sent by $\varphi$ to rings with asphericity $\log R/r + O(u(-\log r))$ and inner radii greater than $r^{1/\gamma}$; this implies \eqref{eq:finer-conformal-structure} translating $-\log r$ into $n$, $\log R/r$ into $k$ and noting that $u(\gamma n)=O(u(n))$ since $u$ is doubling.
Let us prove \eqref{eq:finer-conformal-structure-outer-rings}. Fix $\ell' \in O(u)$ a positive function.
We need $\ell$ such that if $A$ contains a $\ell'$-outer ring then $f(A)$ will contain a $\ell$-outer ring.
Fix $\zeta \in Z$ and $r>0$. Let $r'=\sup \left\{ d(\varphi(\zeta), \varphi(\xi)):d(\zeta,\xi)\leqslant r \right\}$. Let $\xi_0 \in Z$ be such that $d(\varphi(\xi_0),\varphi(\zeta))=r'$. Let $\xi_1 \in Z $ be such that $\varphi(\xi_1) \in B(\varphi(\zeta), r'\exp(\ell'(-\log r'))$. By quasiM{\"o}biusness of $\varphi^{-1}$, there exists $\lambda \in \mathbf{R}_{>0}$ and $v\in O(u)$ a positive function such that
\begin{align*}
    \log^+ \frac{d(\zeta,\xi_1)}{d(\zeta,\xi_0)} & \geqslant \lambda \log^+ \frac{d(\varphi(\zeta),\varphi(\xi_1))}{d(\varphi(\zeta),\varphi(\xi_0))} \\
    & -v(-\log \inf \left\{ d(\vert \varphi(\zeta),\varphi(\xi_0)), d(\vert \varphi(\zeta),\varphi(\xi_1)) \right\}) \\
    & \geqslant \ell'(\lfloor - \log r'\rfloor) -v(-\log r').
\end{align*}
Setting $\ell(n)=\ell'(n/\gamma)+v(n/\gamma)$ this proves \eqref{eq:finer-conformal-structure-outer-rings} for the quasisymmetric structure $\beta$ and the pullback $\varphi^\ast \beta'$ on $Z$. Finally, uniform perfectness of $Z$ and $Z'$ allows to carry the proof up to bounded approximations should certain points not exist.
\end{proof}

\begin{example}[The plane and the twisted plane]
\label{exm:twisted-plane}
Let $Y = \R^2 \rtimes_{\alpha} \R$ and $Y'=\R^2 \rtimes_{\alpha'} \R$ where 
\[ \alpha = \begin{pmatrix} 1 & 0 \\ 0 & 1 \end{pmatrix} \; \text{and} \; \alpha' = \begin{pmatrix} 1 & 1 \\ 0 & 1 \end{pmatrix} \]
and the semi-direct products are formed with $t\in \R$ acting on $\R^2$ as $e^{t\alpha}$ and $e^{t\alpha'}$ respectively.
Equip $Y$ and $Y'$ with left invariant metrics; they are Gromov-hyperbolic and $-t$ is a Busemann function. Identify both Gromov boundaries $\partial_\infty^\ast Y =\partial_\infty Y \setminus [-t]$ and $\partial_\infty^\ast Y'=\partial_\infty Y' \setminus [-t]$ to $\R^2$, {and equip them with the quasisymmetric structures associated to quasidistances $\rho$ and $\rho'$ such that $\rho(e^{t\alpha} \xi_1, e^{t\alpha} \xi_2) = e^t \rho(\xi_1, \xi_2)$ for all $\xi_1, \xi_2 \in \partial_\infty Y$.}
The map $\iota : Y \to Y'$ which is the identity in coordinates is a $O(\log)$-sublinearly biLipschitz equivalence \cite{CornulierCones11}. On $\partial_\infty Y$ and $\partial_\infty Y$ The identity map $\partial_\infty^\ast \iota$ of $\R^2$ is a $O(\log)$-quasisymmetric homeomorphism, as  Figure \ref{fig:example-illustr-thm}.
\end{example}

\subsection{Covering and measures}
\subsubsection{Covering lemma: extracting disjoint balls}

Let $(Z, \beta)$ be a $O(u)$-quasisymmetric structure (Definition \ref{dfn:sublin-conf-struct}) and let $A \in \mathcal{P}(Z)$ be a subset. 
Say that a countable collection of abstract balls $\mathcal B$ is a covering of $A$ if the realizations of the members of $\mathcal B$ cover $A$.
We adapt a classical covering lemma for metric spaces \cite[2.8.4 -- 2.8.8]{FedererGMT}, \cite[p.24]{MattilaSetsMeas}\footnote{We cite both since Federer's statement is more general, but the filtration of balls according to the logarithms of their radii is noticeable in Mattila's proof.} to $O(u)$-quasisymmetric structures; (SC2) may be considered the case with $2$ balls.
The lemma says that out of any covering $\mathcal{B}$ one can extract a disjoint subcovering $\mathcal{C}$ such that $q.\mathcal{C} = \lbrace q.b : b \in \mathcal{B} \rbrace$ is still a covering, where $q$ is a positive function in the $O(u)$-class; for metric spaces it is known as the `` $5r$ covering lemma '' since one can take $5$ as an exponential analog of $q$.

\begin{lemma}
\label{lem:covering-lemma}
Let $(Z, \beta, q, \delta)$ be a $O(u)$-quasisymmetric structure (Definition \ref{dfn:sublin-conf-struct}). Let $A \in \mathcal P (Z)$ be a subset and let $\mathcal{B} \subseteq \beta$ be a countable covering of $A$; assume that $\inf_{\mathcal B} {\delta} > - \infty$.
There exists $\mathcal{C} \subset \mathcal{B}$ such that $q.\mathcal{C}$ covers $A$ and for every $b,b' \in \mathcal{C}$, $\widehat{b} \cap \widehat{b'} = \emptyset$ unless $b =b'$.
\end{lemma}

\begin{proof}
Set $n_0 = \inf_{b \in \mathcal{B}} {\delta}(b)$. 
For every $n \in \Z$, let $\mathcal B_n = \left\{ b \in \mathcal B : {\delta}(b) = n \right\}$.
By induction on $n \in \Z_{\geqslant n_0}$, choose for each $n$ (by Zorn's lemma or Hausdorff's maximality principle, see \cite[0.24]{KelleyGenTop}) a maximal subfamily $\mathcal{C}_n \subset \mathcal{F}_n$ whose realizations are pairwise disjoint and do not intersect the previously chosen balls, that is: 
\begin{itemize}
    \item 
    $\forall (b, b') \in \mathcal{C}_n \times \mathcal{C}_n,\, \widehat{b} \cap \widehat{b'} \neq \emptyset \implies b = b'.$
    \item 
    $\forall b \in \mathcal{C}_n, \,\forall m \in \left\{ n_0, \ldots n-1 \right\},\, \forall b' \in \mathcal{C}_{m},\, \widehat{b} \cap \widehat{b'} = \emptyset$.
    \item 
    $\forall b \in \mathcal{B}_n \setminus \mathcal C_n, \exists b' \in \mathcal{C}_n : \widehat b \cap  \widehat {b'} \neq \emptyset$.
\end{itemize}
By construction, the realizations of members of $\mathcal{C} = \cup_n \mathcal C_n$ are disjoint. 
Let $x \in A$; since $\mathcal B$ covers $A$ there is $b' \in \mathcal B$ such that $\widehat b' \ni x$. Either $b' \in \mathcal{C}$ or, setting $n = {\delta}(b)$, $b' \in \mathcal{B}_n$ and there is $b \in \mathcal C_m$ such that $\widehat{b} \cap \widehat{b'} \neq \emptyset$ with $m \leqslant n$. By (SC2), $\widehat{q.b} \supseteq \widehat{b'}$ so that $\widehat{q.b} \ni x$.
\end{proof}

It follows from the lemma that as soon as a $O(u)$-quasisymmetric structure has a countable covering, then it also has a countable packing $\mathcal{C} \subset \beta$ such that $q.\mathcal{C}$ covers.
This holds for instance, if the quasisymmetric structure comes from a separable metric space.

\subsubsection{Gauges}

Let $(Z, \beta)$ be a $O(u)$-quasisymmetric structure.
We call any function $\phi : \mathcal P(Z) \to [0, + \infty)$ a gauge on $(Z, \beta)$, and we denote by $\mathcal{G}(Z)$ the set of gauges.
For every $\ell \in \mathcal{O}^+(u)$, define a shifted gauge $\widetilde{\phi}^\ell : \mathcal{P}(Z) \to [0, \infty) $ by
\[ \widetilde{\phi}^{\ell}(a) = \sup \lbrace \phi(\widetilde{a}) : (a,\widetilde{a}) \in \mathscr{R}^\ell(\beta) \rbrace. \]

We take the convention that $\widetilde{\phi}^\ell(a)= 0$ if the set on the right-hand side is empty. 
Note that if $\ell \geqslant k$ and $a$ is a $k$-round set, then $\widetilde{\phi}^\ell(a) \geqslant \phi(a)$, for $(a, a)$ is a $\ell$-ring.
It is important that no restriction is made on $\phi$.

\subsubsection{Carath{\'e}odory measures}

Let $(Z, \beta)$ be a $O(u)$-quasisymmetric structure.
For all $k, \ell \in \mathcal{O}^+(u)$, for all $A \in \mathcal{P}(Z)$, define
\begin{align}
    \Phi_{p,(k,n)}(A) & = \inf \left\{  \sum_{b \in \mathscr F} \phi(b)^p : \mathscr{F} \subset \mathscr{B}_n^k(\beta), \vert \mathscr{F} \vert \leqslant \aleph_0,\, \mathscr{F}\, \text{covers} \, A \right\}
    \label{eq:Caratheodory} 
    \\
    \widetilde{\Phi}^{\ell}_{p,(k,n)} (A) & = \inf \left\{  \sum_{b \in \mathscr F} \widetilde{\phi}^\ell(b)^p :  \mathscr{F} \subset \mathscr{B}_n^k(\beta), \vert \mathscr{F} \vert \leqslant \aleph_0,\, \mathscr{F}\, \text{covers} \, A \right\}
    \label{eq:shifted-Caratheodory}
\end{align}
and $\Phi_{p;k}(A) = \lim_{n \to +\infty} \Phi_{p,(k,n)}(A)$, $\widetilde \Phi_{p;k}^\ell(A) = \lim_{n \to +\infty} \widetilde{\Phi}_{p,(k,n)}^{\ell}(A)$.
The $O(u)$-quasisymmetric structure $\beta$ is not specified, however if $\beta$ and $\beta'$ are two equivalent $O(u)$-quasisymmetric structures on $Z$ and if $\lambda$, $\eta$, $\eta'$ are such that any $(\ell,n)$-ring for $\beta$ (resp. for $\beta'$) is a $(\eta(\ell), \lfloor \lambda n \rfloor)$-ring for $\beta'$ (resp. for $\beta$), then denoting $\Phi$ and $\Phi'$ the measures that correspond to $\phi$ for $\beta$ and $\beta'$ then
\begin{equation}
    \label{eq:comparison-Caratheordory-measures}
    \left( \Phi '\right)_{p;\eta(k)} \leqslant \Phi_{p;k} \; \text{and} \; \left( \widetilde \Phi '\right)^{\ell}_{p;\eta(k)} \leqslant \widetilde \Phi^{\eta'(\ell)}_{p;k}
\end{equation}
since any covering by $(k,n)$ round sets with respect to $\beta$ is a covering by $(\eta(k), \lfloor \lambda n \rfloor)$ round sets with respect to $\beta'$, and any $\ell$-ring with respect to $\beta'$ is a $\eta'(\ell)$-ring with respect to $\beta$ (note that $\eta$ or $\eta'$ appears on superscript when on the right of $\leqslant$ and on subscript when on the left).

\begin{lemma}[Comparisons with Hausdorff measures]
Assume that the quasisymmetric structure $\beta$ is that of a metric space as in Example \ref{exm:space-with-a-quasidist}. Let $s\in \R_{>0}$ and $\phi(a) = \operatorname{diam}(a)^s$. Then for every $p \in (0,+\infty)$, $k \in \mathcal{O}^+(u)$ and $\varepsilon\in (0,sp)$,
\begin{equation}
    \label{eq:comparison-with-Hausdorff-measures-1}
    \mathcal{H}^{sp+\varepsilon} \ll \Phi_{p;k} \ll \mathcal{H}^{sp-\varepsilon}.
\end{equation}
Moreover, for all $\ell \in \mathcal{O}^+(u)$
\begin{equation}
    \label{eq:comparison-with-Hausdorff-measures-2}
    \widetilde \Phi_{p;k}^\ell  \ll \mathcal{H}^{sp-\varepsilon}.
\end{equation}
\end{lemma}

\begin{proof}
Since balls are $(k,n)$-round sets and diameters of $(k,n)$-round sets are bounded by $e^{-n/2}$ for large enough $n$, $\mathcal{H}^{sp} \leqslant \Phi_{p;k} \leqslant \mathcal{S}^{sp}$, where $\mathcal{H}$ and $\mathcal S$ denote the (non-normalized) Hausdorff and spherical Hausdorff measures. As $\mathcal S^{sp} \leqslant 2^{sp} \mathcal{H}^{sp}$, the comparisons \eqref{eq:comparison-with-Hausdorff-measures-1} follow.
As for \eqref{eq:comparison-with-Hausdorff-measures-2}, note that if $a$ is a ball then $\log \widetilde{\phi}^\ell(a) \leqslant s(\log \operatorname{diam}(a) + \ell (\lceil - \log \operatorname{diam}(a) \rceil))$, especially as $\ell$ is sublinear, for every $\varepsilon >0$, for small enough balls $a$, $\widetilde{\phi}^\ell (a) \leqslant \operatorname{diam}(a)^{s- \varepsilon /(2p)}$, and then $\widetilde{\Phi}_{p;k}^\ell \leqslant \mathcal{S}^{p- \varepsilon/2} \ll \mathcal{H}^{sp - \varepsilon}$.
\end{proof}

\subsubsection{Packing Pre-measure}

Let $(Z, \beta)$ be a quasisymmetric structure and let $A \in \mathcal{P}(Z)$ be a subset.
Let $\mathscr{P}$ be a countable collection of $(k,n)$-outer rings; say that $\mathscr{P}$ is a $(k,n)$-packing centered on $A$, denoted $\mathscr{P} \in \operatorname{Packings}_{k,n}(A)$ if inner sets meet $A$ and outer sets are disjoint; formally
\begin{itemize}
    \item 
    For every $\mathbf{a}=(a^-, a^+)$ in $\mathscr{P}$, $a^- \cap A \neq \emptyset$.
    \item
    For every $\mathbf{a}_0$, $\mathbf{a}_1$ in $\mathscr{P}$, $a_0^+ \cap a_1^+ \neq \emptyset \implies \mathbf a_0 = \mathbf a_1$.
\end{itemize}
Similarly to the shifted packing measure $\widetilde{\Phi}$, define a shifted packing pre-measure
\begin{equation}
    \mathrm{P}\widetilde{\Phi}_{p;k}^{\ell}(A) = \lim_{n \to +\infty} \sup \left\{ \sum_{\mathbf a \in \mathscr{P}} \widetilde{\phi}^\ell(a^-)^p: \mathscr{P}\in \operatorname{Packings}_{k,n}(A) \right\}
    \label{eq:dfn-shifted-packing-meas}
\end{equation}
or $0$ if there exists no packing indexing the sums.

\begin{remark}
    \label{rem:pros-and-cons-of-packing-measures}
    Let $\phi = \lambda \cdot {}^0 \phi+ {}^1\phi$ with $\lambda \in \R_{\geqslant 0}$ and ${}^i \phi \in \mathcal{G}(Z)$ for $i \in \lbrace 0, 1 \rbrace$.
    Associate ${}^i \operatorname{P} \widetilde{\Phi}^\ell_{p;k}$ to ${}^i \phi$ by \eqref{eq:dfn-shifted-packing-meas}.
    {Then by the Minkowski inequality 
    \begin{equation}
        \left( \operatorname{P} \widetilde{\Phi}^\ell_{p;k} \right)^{1/p} \leqslant \lambda \cdot \left( {}^0 \operatorname{P} \widetilde{\Phi}^\ell_{p;k} \right)^{1/p} + \left( {}^1 \operatorname{P} \widetilde{\Phi}^\ell_{p;k} \right)^{1/p}.
        \label{eq:application-of-minkowski-ineq}
    \end{equation}}
\end{remark}

\begin{remark} When changing $O(u)$-quasisymmetric structure from $\beta$ to $\beta'$, the analogs of the comparisons \eqref{eq:comparison-Caratheordory-measures} are
    \begin{equation}
    (\mathrm{P} \widetilde\Phi')_{p;k}^{\eta(\ell)} \geqslant
        \mathrm{P} \widetilde\Phi_{p;\overline{\eta}(k)}^{\ell}.
    \label{eq:comparison-packing-covering-measures}
    \end{equation}
Indeed \eqref{eq:finer-conformal-structure-outer-rings} implies that $\operatorname{Packings}_{\overline{\eta}(k'),n}(\beta) \subset \operatorname{Packings}_{k',\lfloor \lambda n \rfloor} (\beta')$ whereas, every $\ell$-ring for $\beta$ being a $\eta(\ell)$-rings with respect to $\beta'$, the supremum in \eqref{eq:dfn-shifted-packing-meas} is taken over larger sums.
\end{remark}

\begin{remark}
Pansu uses a notion of packing with bounded multiplicity \cite{PansuQuasisym}. However it is not convenient here because even on doubling spaces, if $b\in \beta$ is such that $\delta(b)=n$ then $\widehat{\ell.b}$ cannot be covered by a uniformly bounded number (that is, a number independent of $n$) of concrete balls of the form $\widehat{b'}$ with $\delta(b')=n$.
\end{remark}

\section{Conformal invariants}
\label{sec:conformal-invariants}

By conformal invariants we mean real numbers attached to $O(u)$-quasisymmetric structures, possibly parametrized (for instance by asphericities) and respecting invariance under conformal equivalence. 
This invariance should not be understood too strictly: the vanishing, or infinitude, for some choice of parameters is considered an invariant, though those parameters may vary. 

\subsection[Moduli and variations]{Combinatorial moduli and functions of bounded $p$-variation}

\subsubsection[Combinatorial moduli]{Carath{\'e}odory and packing combinatorial moduli}

The modulus is obtained by minimizing $\widetilde{\Phi}$ under a normalization constraint on the gauge functions, compare Pansu {\cite[2.4]{PansuDimConf}} and Tyson \cite[3.23]{TysonQuasiconfQuasisymmm}: all members of $\Gamma$ should have measure (to be thought of as a length\footnote{This is similar in spirit to requiring a Riemannian metric in a given conformal class to confer sufficient length to any curve in a family as in the definition of the classical moduli.}) greater than $1$.

\begin{definition}
\label{dfn:modulus}
    Let $\Gamma$ be a family of subsets in a conformal structure $(Z, \beta)$, $p \in (0, + \infty)$, $k$, $\ell$ and $m$ in $\mathcal{O}^+(u)$. Define
    \begin{equation*}
    \operatorname{mod}^{\ell,m}_{p;k}(\Gamma, \beta) = \inf \left\{ \widetilde{\Phi}^\ell_{p;k} (Z) : \phi \in \mathcal{G}_{m}(\Gamma, \beta) \right\} \; \text{and}
    \end{equation*}
    \begin{equation*}
        \operatorname{pmod}^{\ell,m}_{p;k}(\Gamma, \beta) = \inf \left\{ \mathrm{P} {\widetilde \Phi}_{p;k}^{\ell} (Z) : \phi \in \mathcal{G}_{m}(\Gamma, \beta) \right\},
    \end{equation*}
where $\mathcal{G}_m(\Gamma, \beta) = \lbrace \phi \in \mathcal{G}(\beta) : \forall \gamma \in \Gamma, \Phi_{1;m}(\gamma) \geqslant 1 \rbrace$ is called a set of admissible gauges for $\Gamma$.
\end{definition}

\begin{remark}
The notation for moduli is the standard one (cf. \cite{PansuDimConf} and \cite{TysonQuasiconfQuasisymmm}), up to the position of upper/lower indices. This is in order to emphasize the monotonicity with respect to the parameters: $\operatorname{mod}^{\ell,m}_{p;k}$ increases with $\ell$ and $m$ but decreases with $p$ and $k$. We shall observe the same convention with other forthcoming quantities.
\end{remark}

When changing conformal structure, the moduli change in the following way:

\begin{lemma}[compare {\cite[2.6]{PansuDimConf}}]
\label{prop:modulus-conf-inv}
{Let $\beta$ and $\beta'$ be two $O(u)$-equivalent $O(u)$-quasisymmetric structures} on $Z$. Let $\Gamma \subset \mathcal{P}(Z)$. Set $\eta, \eta'$ and $\overline{\eta}$ so that $\mathscr{R}^k(\beta) \subset \mathscr{R}^{\eta(k)}(\beta')$, $\mathscr{R}^k(\beta') \subset \mathscr{R}^{\eta'(k)}(\beta)$ and $\mathscr{O}_{\overline{\eta}({j'})}(\beta) \subset \mathscr{O}_{j'}(\beta')$ for every $k,j' \in \mathcal{O}^+(u)$. 
Then for every $k,\ell,m \in \mathcal{O}^+(u)$,
\begin{align}
    \label{eq:comparison-moduli}
    \operatorname{mod}_{p;\eta'(k)}^{\ell,m} \left( \Gamma, \beta \right) & \leqslant  \operatorname{mod}_{p;k}^{\eta(\ell),\eta(m)}(\Gamma, {\beta'}) \; \text{and} \\
    \label{eq:comparison-packing-moduli}
    \operatorname{pmod}_{p;\overline{\eta}(k)}^{\ell,m} \left( \Gamma, \beta \right) & \leqslant  \operatorname{pmod}_{p;k}^{\eta(\ell),\eta(m)}(\Gamma, {\beta'}).
\end{align}
\end{lemma}

\begin{proof}
Let us first concentrate on the change of admissible gauges.
Recall that if $a$ is a round set, then $(a,a)$ is a ring; hence, by assumption, $m$-round sets for $\beta'$ are $\eta(m)$-round sets for $\beta'$. In view of \eqref{eq:Caratheodory}, for all $\gamma$ in $\Gamma$, $\Phi_{1;m}(\gamma) \geqslant \left( \Phi' \right)_{1;\eta(m)}(\gamma)$ (the infimum being computed on more coverings, is smaller), especially $\left( \Phi' \right)_{1;\eta(m)}(\gamma) \geqslant 1$ implies $\Phi_{1;m}(\gamma) \geqslant 1$ so that
\begin{equation}
    \label{eq:comparison-admissible-gauges}
    \mathcal{G}_{\eta(m)}(\beta') \subseteq \mathcal{G}_m(\beta).
\end{equation}
    Now, by \eqref{eq:comparison-Caratheordory-measures}, for all $\phi \in \mathcal{G}(Z)$, $\left( \widetilde \Phi' \right)^{\eta(\ell)}_{p;k} \geqslant \widetilde \Phi_{p; \eta'(k)}^\ell$. 
    Hence, on the left-hand side of \eqref{eq:comparison-moduli}, the infimum in Definition \ref{dfn:modulus} is taken over more gauges, while common admissible gauges contribute to lower values, than on the right-hand side.
    
    The proof of \eqref{eq:comparison-packing-moduli} follows the same lines starting from \eqref{eq:comparison-admissible-gauges}, but uses \eqref{eq:comparison-packing-covering-measures} to compare the shifted packing measures instead of \eqref{eq:comparison-Caratheordory-measures}. 
\end{proof}

\subsubsection{Functions of locally bounded $p$-variation}

We investigate here function spaces that are carried by $O(u)$-quasisymmetric homeomorphisms with a shift in an asphericity parameter.
The notion of $p$-variation we use here is inspired by Pansu's \cite[6.1]{PansuDimConf} (Beware that Pansu calls it ``energy'') but it is actually more closely related to Kleiner and Xie's $Q$-variation (\cite[Definition 3.2]{XieQStilted}, \cite[4]{XieLargeScale}).
For quasimetric spaces $Z$ and $u=1$, the $p$-variation we define is Kleiner and Xie's $Q$-variation, and the reader familiar with $Q$-variation may translate $\mathbf V_{p;k}^\ell(f)(-)$ into $V_{Q,K}(f_{\mid -})$ with $\ell=0$, $p=Q$ and $k =\log K$.

Let $\beta$ be a $O(u)$-quasisymmetric structure on a set $Z$, and let $f : Z \to \C$ be a continuous function. 
Given $p \in [1,+\infty)$ and $k, \ell \in \mathcal{O}^+(u)$ one can associate to $f$ a pre-measure on $Z$ by 
$\mathbf V_{p;k}^{\ell}(f) =  \mathrm{P} \widetilde{\Phi}_{p;k}^{\ell}$ using the gauge 
\[ \phi(a) = \operatorname{diam} f(a) =: \operatorname{osc}(f,a). \]
Fix $k,\ell \in \mathcal{O}^+(u)$.
Say that a continuous function $f$ has bounded $(p;k,\ell)$-variation if  $\mathbf V_{p;k}^{\ell}(f)$ is locally finite.
If $\Omega \subset Z$ is an open subset, denote the space of functions of bounded $p$-variation by
\begin{equation*}
\mathscr{W}^{p;k}_{\ell;\mathrm{loc.}}(\Omega)  = \Big\{  f \in \mathscr{C}^0(\Omega) : \forall K \in \mathcal{P}(\Omega), K \Subset \Omega \implies  \mathbf V_{p;k}^{\ell}(f)(K) < + \infty  \Big\}.
\end{equation*}
For all $K$ compact in $\Omega$, $k \in \mathcal{O}^+(u)$, $\ell \in \mathcal{O}^+(u)$ and $p \in [1,+\infty)$ define
\begin{equation}
    \label{eq:definition-Besov-seminorm}
    \Vert f \Vert^{K;\ell}_{p;k} := \Vert f \Vert_{C^0(K)} + \mathbf V_{p;k}^{\ell}(f)(K)^{1/p}.
\end{equation}
{\begin{lemma}
Let $\Omega$ be an open subset of a $O(u)$-quasisymmetric structure $\beta$. For every $p \in [1, + \infty)$ and $\ell, k \in \mathcal O^+(u)$, $\mathscr{V}^{p;k}_{\ell; \mathrm{loc.}}(\Omega)$ is an algebra for pointwise multiplication and for every $K \Subset \Omega$, $f \mapsto \Vert f \Vert^{K;\ell}_{p;k}$ defines a multiplicative seminorm on $\mathscr{V}^{p;k}_{\ell;\mathrm{loc.}}(\Omega)$.
\end{lemma}

\begin{proof}
By \eqref{eq:application-of-minkowski-ineq} and the triangle inequality in $\mathbf R$, for any $f,g \in \mathscr C(\Omega)$ and $\lambda \in \C$ one has $\mathbf V_{p;k}^\ell(\lambda f+ g) \leqslant \vert \lambda \vert \mathbf V_{p;k}^\ell(f) +  \mathbf V_{p;k}^\ell(g))$, so that $\mathscr{V}^{p;k}_{\ell;\mathrm{loc.}}(\Omega)$ is a vector space.
Further, for every $A \subseteq \Omega$,
\begin{equation}
    \mathrm{osc}(fg,a) \leqslant \sup_A \vert f \vert \operatorname{osc}(g,a) + \sup_A \vert g \vert \operatorname{osc}(f,a),
    \label{eq:inequality-on-oscillations}
\end{equation}
while, by definition
\begin{equation}
    \mathbf V_{p;k}^\ell(fg)(K) = \lim_{n \to + \infty} \sup_{\mathscr{P} \in \operatorname{Packings}_{k,n}(K)} \sum_{\mathbf a \in \mathscr{P}} \sup_{(a^-, A) \in \mathscr{R}^\ell(\beta)} \mathrm{osc}(fg,A)^p.
    \label{eq:definition-of-energy-of-g}
\end{equation}
At this point, note that since $K$ has been assumed compact, since the topology associated to $\beta$ is uniform, since $f$ is continuous and since $\limsup_n \cup_{\mathbf a \in \mathscr{P}_n} \sup_{(a^-, A) \in \mathscr{R}^\ell(\beta)} A \subseteq K$ for every sequence of $(k;n)$ packings $\mathscr{P}_n$,
\begin{equation*}
    \lim_{n \to + \infty} \sup_{\mathscr{P} \in \operatorname{Packings}_{k,n}(K)} \sup_{\mathbf a \in \mathscr P} \sup_{(a^-, A) \in \mathscr{R}^\ell(\beta)} \sup_A \vert f \vert \leqslant \Vert f \Vert_{C^0(K)}
\end{equation*}
and the same inequality holds for $g$ so that inserting \eqref{eq:inequality-on-oscillations} in \eqref{eq:definition-of-energy-of-g} and letting $n \to + \infty$ using this estimate and the Minkowski inequality yields
\begin{equation}
    \mathbf V_{p;k}^\ell(fg)(K)^{1/p} \leqslant \Vert f \Vert_{C^0(K)} \mathbf V_{p;k}^\ell(g)^{1/p} + \Vert g \Vert_{C^0(K)} \mathbf V_{p;k}^\ell(f)^{1/p}.
    \label{eq:bound-on-energy-fg}
\end{equation}
From there (recall that $\Vert - \Vert^{K;\ell}_{p;k}$ was defined in \eqref{eq:definition-Besov-seminorm}),
\begin{align*}
\Vert fg \Vert^{K;\ell}_{p;k} & =  \Vert fg \Vert_{C^0(K)} + \mathbf V_{p;k}^{\ell}(fg)(K)^{1/p} \\
& \overset{\eqref{eq:bound-on-energy-fg}}{\leqslant} \Vert f \Vert_{C^0(K)} \Vert g \Vert_{C^0(K)} +  \Vert f \Vert_{C^0(K)} \mathbf V_{p;k}^\ell(g)^{1/p} + \Vert g \Vert_{C^0(K)} \mathbf V_{p;k}^\ell(f)^{1/p} \\
& \leqslant \left( \Vert f \Vert_{C^0(K)} + \mathbf V_{p;k}^{\ell}(f)(K)^{1/p} \right)  \left( \Vert g \Vert_{C^0(K)} + \mathbf V_{p;k}^{\ell}(g)(K)^{1/p} \right) \\
& = \Vert f \Vert^{K;\ell}_{p;k} \Vert g \Vert^{K;\ell}_{p;k}. \qedhere
\end{align*}
\end{proof}
}
In order to add structure to $\mathscr{V}^{p;k}_{\ell;\mathrm{loc.}}(\Omega)$, we will need to assume more on the topology associated with $\beta$.

\begin{definition}[hemicompactness]
Let $X$ be a Hausdorff topological space. An admissible exhaustion of $X$ is an increasing sequence of compact subspaces $(K_n)_{n \geqslant 0}$ of $X$ such that for every compact $K$ of $X$ there exists $n$ such that $K \subset K_n$. A space is hemicompact if it has an admissible exhaustion.
\end{definition}

If $Z$ is a locally compact, second countable topological space, then any open subset of $Z$ is hemicompact. Indeed by Lindel{\"o}f's lemma in a second countable space, every open subset is a Lindel{\"o}f space (meaning that any open cover of it has a countable subcover) \cite[Chapter 1, Theorem 15]{KelleyGenTop}, and a locally compact Lindel{\"o}f space is hemicompact.

\begin{lemma}
Let $(Z,\beta)$ be a $O(u)$-quasisymmetric structure with locally compact, secound countable topology.
For all non-empty open $\Omega \subset Z$, 
$\mathscr{V}^{p;k}_{\ell;\mathrm{loc.}}(\Omega)$ defines a unital commutative algebra with a topology defined by a countable family of seminorms. Further, if $\varphi:(Z',\beta') \to (Z, \beta)$ is a $O(u)$-quasisymmetric homeomorphism then for every open $\Omega' \subset Z'$, letting $\Omega =\varphi(\Omega')$ {the identity map defines} linear continuous algebra homomorphisms
\begin{equation}
    \label{eq:continuous-embeddings-algebras}
    \mathscr{V}^{p;k}_{\eta \circ \eta'(\ell);\mathrm{loc.}} (\Omega') \hookrightarrow
    \varphi^\ast \mathscr{V}^{p;\overline{\eta}(k)}_{\eta'(\ell);\mathrm{loc.}}(\Omega) \hookrightarrow
    \mathscr{V}^{p;\overline{\eta}'\circ \overline{\eta}(k)}_{\ell;\mathrm{loc.}}(\Omega').
\end{equation}
\end{lemma}

\begin{proof}
By the observation above each open subset $\Omega$ being hemicompact, has an admissible exhaustion $(K_n)$.
The $\left( \Vert - \Vert^{K_n,k}_{\ell} \right)$ for an exhaustion $K_n$ define a countable family of seminorms on $\mathscr{V}^{p;k}_{\mathrm{loc.}}$; the hemicompactness ensures that the topology does not depend on the choice of the sequence $(K_n)$.
To prove the part about $O(u)$-quasisymmetric homeomorphisms we can assume that $\beta'$ is an $O(u)$-equivalent structure on the same set $Z$.
Denote $\mathbf V$, resp.\ $\mathbf V'$ the variations computed with respect to $\beta$, resp.\ $\beta'$. 
By \eqref{eq:comparison-packing-covering-measures},
\begin{equation*}
    \forall k,\ell \in \mathcal{O}^+(u), \mathbf
    E_{p;k}^{\eta'(\ell)} (f) \geqslant (\mathbf V')_{p;\overline{\eta}'(k)}^\ell (f)
\end{equation*}
(this may be compared to Xie \cite[Lemma 3.1]{XieQStilted})
so that $\mathscr{V}^{p;k}_{\mathrm{loc.}}(\Omega, \beta_{\mid \Omega})$ continuously embeds in $\mathscr{V}^{p;\overline{\eta}'(k)}_{\mathrm{loc.}}(\Omega \mid \beta'_{\mid \Omega})$. \eqref{eq:continuous-embeddings-algebras} is obtained by applying this twice and reversing the r{\^o}les of $\beta$ and $\beta'$.
\end{proof}

\subsubsection{Condensers and capacities}

For $p \in [1, + \infty)$, $k,\ell \in \mathcal{O}^+(u)$ and $\Omega$ an open subset in a $O(u)$-quasisymmetric structure, denote by ${\mathscr{V}}^{p;k}_{\ell;\mathrm{loc.}}(\Omega,\R)$ the $\R$-subspace of $\mathscr{V}^{p;k}_{\ell;\mathrm{loc.}}(\Omega)$ of $\R$-valued functions. 

\begin{definition}[Condenser, capacity]
Let $Z$ be a $O(u)$-quasisymmetric structure and let $\Omega$ be an open subspace.
A condenser in $\Omega$ is a triple of subspaces $(C,\partial_0 C, \partial_1 C)$ such that $C$ is relatively compact, $\partial_0 C$ and $\partial_1 C$ are closed disjoint, and contained in $\overline{C} \setminus C$. Its capacity is
\[ \operatorname{cap}_{p;k}^{\ell}(C) = \inf \left\{ \mathbf V_{p;k}^{\ell}(f)(C) :\, f \in \mathscr V^{p;k}_{\ell;\mathrm{loc.}}(\Omega,\R), \, f_{\mid \partial_0 C} \leqslant 0,\, f_{\mid \partial_1 C} \geqslant 1 \right\}. \]
\end{definition}

\begin{lemma}
\label{lem:capacities-and-moduli}
    Let $(C, \partial_0 C, \partial_1 C)$ be a condenser in $\Omega$, open subset of a $O(u)$-quasisymmetric structure $\beta$. For all $k,\ell,m \in \mathcal{O}^+(u)$, if $\Gamma$ is any family of curves joining $\partial_0 C$ and $\partial_1 C$ in $C$ then
    \begin{equation}
        \operatorname{pmod}^{\ell,m}_{p;k}(\Gamma) \leqslant \operatorname{cap}_{p;k}^{\ell}(C).
        \label{eq:comparison-packing-moduli-capacity}
    \end{equation}
\end{lemma}

\begin{proof}
Let $\varepsilon > 0$.
Let $f \in \mathscr{V}^{p;k}_{\ell;\mathrm{loc.}}(\Omega, \R)$ be such that $f_{\mid \partial_0 C} \leqslant 0, f_{\mid \partial_1 C} \geqslant 1$.
Let us prove that the gauge $\phi :a \mapsto \mathrm{diam} f(a)$ is in $\mathcal{G}_m(\Gamma,\beta)$; the conclusion will follow by applying the definition of capacites and $p$-variations.
By the intermediate value theorem, for every $\gamma$ in $\Gamma$, $f(\gamma)$ contains $[0,1]$.
Consequently, whenever $\mathscr{F}$ is a covering of $\gamma$ by $m$-round sets, by countable subadditivity of the outer measure $\mathcal{H}^1$ on $\R$
\begin{align*}
    \sum_{a \in \mathscr{F}} \phi(a) = \sum_{a \in \mathscr F} \operatorname{diam} f(a) = \sum_{a \in \mathscr F} \mathcal{H}^1 \operatorname{Conv}(f(a)) 
    & \geqslant \sum_{a \in \mathscr F} \mathcal{H}^1 f(a) \\
    & \geqslant \mathcal H^1 \left( \bigcup_{a \in \mathscr F} f(a) \right) \geqslant 1. \qedhere
\end{align*}
\end{proof}

\subsection{Diffusivity}

The following is a central result in conformal dimension theory 
\cite[4.1.3]{MTconfdim}. 
The guiding principle is a length-volume estimate for a Riemannian parallelotope {\cite[2.2]{PansuDimConf}}; in order to transpose this to the combinatorial moduli, one has to retain a diffusivity condition expressing that a family of curves is sufficiently spread out in the space, \eqref{ineq:diffusivity-assumption} below. 
We give two variants: the first is Pansu's original; the second one is a packing variant.

\subsubsection{Carath{\'e}odory variant}

\begin{proposition}
\label{prop:diffusivity}
Let $(Z, \beta, {\delta}, q)$ be a $O(u)$-quasisymmetric structure. Let $\Gamma$ be a collection of subsets in $Z$, endowed with a positive measure $d\gamma$ such that for any $b \in \beta$, $\lbrace \gamma \in \Gamma : \gamma \cap \widehat b \neq \emptyset \rbrace$ is measurable. For each $\gamma \in \Gamma$, let $m_\gamma$ be a probability Borel measure on $\gamma$. Let $p \in (1, + \infty)$. Assume that there exists a constant $\tau \in (0, +\infty)$ and $r \in \mathcal{O}^+(u)$ such that
\begin{equation}
\tag{$\mathrm D(p,r)$}
\limsup_{n \to + \infty} \sup_{\underline b \in \beta : {\delta}(\underline b) \geqslant n} \int_\Gamma m_\gamma(\gamma \cap \widehat{r.\underline b})^{1-p} \mathbf{1}_{\gamma \cap \widehat{\underline b} \neq \emptyset} d \gamma \leqslant \tau.
\label{ineq:diffusivity-assumption}
\end{equation}
Then for every $k,m \in \mathcal{O}^+(u)$,
\begin{equation}
    \label{eq:diffusivity-caratheordory-conclusion-bound-moduli}
    \operatorname{mod}^{\ell,m}_{p;k}(\Gamma) \geqslant \frac{1}{\tau} \int_\Gamma d\gamma,
\end{equation}
where\footnote{The conclusion of the lemma (as the assumption \eqref{ineq:diffusivity-assumption} is all the more weaker that $r$ is large. In subsection \ref{subsec:heintze} we can arrange the quasisymmetric structure so that $r$ can be assumed $1$, however in subsection \ref{subsec:fuchsian} it is really necessary.} $\ell = q \dotplus r \dotplus k$ {(We recall that the operation $\dotplus$  was defined in \ref{def-of-dotplus})}.
\end{proposition}

\begin{proof}
Up to the formalism, the proof is due to Pansu \cite[2.9]{PansuDimConf} and we do not depart from it. Inequality \eqref{eq:diffusivity-caratheordory-conclusion-bound-moduli} will actually be obtained through a stronger one: for any $0$-admissible gauge $\phi$, 
\begin{equation}
    \widetilde{\Phi}_{p;k}^{\ell}(\Gamma) \geqslant \tau^{-1} \int_{\Gamma} \Phi_{1;0}(\gamma)^p d \gamma.
    \label{eq:diffusivity-caratheordory-conclusion-bound-moduli-strong}
\end{equation}
(To see why \eqref{eq:diffusivity-caratheordory-conclusion-bound-moduli-strong} implies \eqref{eq:diffusivity-caratheordory-conclusion-bound-moduli} with $m=0$ note that since $p \geqslant 1$ and $\phi$ is admissible the right-hand side is greater than $\int_\Gamma d\gamma$; finally $\operatorname{mod}^{\ell,m}_{p;k}$ increases with $m$). Set an admissible gauge $\phi$. 
Define, for all $n$,
\[ \tau_n :=  \sup_{\underline b \in \beta : {\delta}(b) \geqslant n} \int_\Gamma m_\gamma(\gamma \cap \widehat{r.\underline b})^{1-p} \mathbf{1}_{ \gamma \cap \widehat {\underline b} \neq \emptyset } d \gamma. \]
Fix $n \in \Z$. Let $k \in \mathcal{O}^+(u)$.
Let $\mathscr{F}$ be a countable covering of $Z$ by $(k,n)$-round sets of $\beta$; taking inner ball $\underline{b} \in \beta$ for each round set $b \in \mathscr{F}$ gives a countable $\mathcal{B} \subset \beta$ such that $k.\mathcal{B}$ covers $Z$.
For $\gamma \in \Gamma$ define $\mathcal{B}_\gamma = \lbrace \underline b \in \mathcal{B} : b \cap \gamma \neq \emptyset \rbrace$.
For every $\gamma$, 
$k.\mathcal{B}_\gamma$ is a covering of $\gamma$, since every $x \in \gamma$ is contained in a $b \in \mathscr{F}$ such that $\underline{b}$ has been selected in $\mathcal{B}_\gamma$.
All the more, $r.k.\mathcal{B}_\gamma$ is a covering of $\gamma$ and by Lemma \ref{lem:covering-lemma} one can extract $\mathcal{C}_\gamma$ from $\mathcal{B}_\gamma$ such that $q.r.k.\mathcal{C}_\gamma$ covers $\gamma$ and have disjoint realizations.
Note that $(\widehat{\underline b}, \widehat{q.r.k.\underline b}) \in \mathscr{R}^{q \dotplus r \dotplus k}(\beta)$ (as $\delta (q.r.k.b) = \delta(b) - (q \dotplus r \dotplus k)(\delta(b))$), hence 
\begin{equation*}
\phi(\widehat{q.r.k.\underline b}) \leqslant \sup \left\{ \phi(\widetilde b) :(b,\widetilde{b}) \in \mathscr{R}^\ell(\beta) \right\} 
= \widetilde \phi^\ell(b).
\end{equation*}
Recall that $q.r.k.\mathcal{B}_\gamma$ covers $\gamma$. 
Thus
\begin{equation}
     \Phi_{1,(0,n - \ell (n))}(\gamma) \leqslant \sum_{\underline b \in \mathcal{C}_\gamma} \phi \left( \widehat{q.r.k.\underline{b}} \right)
       \leqslant \sum_{\underline b \in \mathcal{C}_\gamma} \widetilde \phi^\ell(b).  
     \label{eq:diffusivite-bound-by-sum-gamma-shift}
\end{equation}
Next, apply H{\"o}lder's inequality to $\alpha, \zeta : \mathcal{C}_\gamma \to \R$ defined by
\[
\alpha(\underline b)  =  \widetilde{\phi}^\ell(\widehat{\underline b}) m_\gamma(\widehat{r.k.\underline{b}} \cap \gamma)^{(1-p)/p} \; \text{and} \;
\zeta({\underline b})  = m_\gamma(\widehat{r.k.b} \cap \gamma)^{(p-1)/p}
\]
so that
\begin{align}
\Phi_{1,(0,n - \ell(n))}(\gamma)^p 
& \underset{\eqref{eq:diffusivite-bound-by-sum-gamma-shift}}{\leqslant} 
\left( \sum_{b \in \mathcal{C}_\gamma} \alpha(\underline b)^p \right) \left( \sum_{b \in \mathcal{C}_\gamma} \zeta(\underline b)^{p/(p-1)} \right)^{p-1} \notag \\
& \leqslant \left( \sum_{b \in \mathcal C_\gamma} \widetilde{\phi}^{\ell}(b)^p m_\gamma(\widehat{r.k.\underline b} \cap \gamma)^{1-p} \right) \left( \sum_{b \in \mathcal C_\gamma} m_\gamma(\widehat{r.k.\underline b} \cap \gamma) \right)^{p-1} \notag \\
& \leqslant \left( \sum_{b \in \mathcal C _ \gamma}  \widetilde{\phi}^{\ell}(b)^p m_\gamma(\widehat{r.k.\underline b} \cap \gamma)^{1-p} \right) \left( m_\gamma(\gamma) \right)^{p-1}.
\label{eq:diffusivity-2}
\end{align}
The last inequality comes from the fact that the $\widehat{r.k.\underline{b}}$ for $b \in \mathcal{C}_\gamma$ are disjoint by construction, hence their intersections with $\gamma$ are disjoint, and $m_\gamma$ is subadditive. 
Further, since $m_\gamma$ is a probability measure, \eqref{eq:diffusivity-2} rewrites
\begin{equation*}
    \Phi_{1,(0,n -  \ell(n))}(\gamma)^p \leqslant \sum_{b \in \mathcal C_ \gamma}  \widetilde{\phi}^{\ell}(b)^p m_\gamma(\widehat{r.k.\underline b} \cap \gamma)^{1-p}.
\end{equation*}
Integrating over $\Gamma$ yields
\begin{align*}
    \int_{\Gamma} \Phi_{1,(0,n -  \ell(n))}(\gamma)^p d \gamma & \leqslant \int_{\Gamma} \sum_{\underline b \in \mathcal{C}_\gamma} \widetilde{\phi}^{\ell}(b)^p m_\gamma (\widehat{r.k.\underline{b}} \cap \gamma)^{1-p} d \gamma \\
    & \leqslant \sum_{b \in \mathscr F} \widetilde{\phi}^{\ell}(b)^p \int_{\Gamma} \mathbf{1}_{\underline b \in \mathcal{C}_\gamma} m_\gamma (\widehat{r.k.\underline{b}} \cap \gamma)^{1-p} d \gamma 
    \leqslant \tau_n \sum_{b \in \mathscr{F}} \widetilde{\phi}^{\ell}(b)^p .
\end{align*}
Infimizing over every countable $\mathscr{F} \subset \mathscr B^k_{n}(\beta)$ that covers $X$ one obtains:
\begin{equation}
     {\widetilde \Phi}_{p,(k,\ell;n - \ell(n))}(X) \geqslant \tau_n^{-1} \int_{\Gamma} \Phi_{1,(0,n- \ell(n))}(\gamma)^p d \gamma.
     \label{eq:in-diffusivity-lower-bound-Phi-tilde}
\end{equation}
By monotone convergence, if $\phi \in \mathcal{G}_m(\beta)$ then 
\[ \lim_{n \to + \infty} \int_{\Gamma} \Phi_{1,(0,n- \ell(n))}(\gamma)^p d \gamma =\int_{\Gamma} \Phi_{1,(0,n- \ell(n))}(\gamma)^p \geqslant \int_\Gamma d \gamma >0. \]

Since $\ell$ is sublinear, $n - \ell(n)$ goes to $+\infty$ as $n \to +\infty$. 
Especially, ${\widetilde \Phi}^{\ell}_{p;(k,n)}(X)$ is bounded below by \eqref{ineq:diffusivity-assumption}. The conclusion is reached by applying the Definition \ref{dfn:modulus} of the modulus.
\end{proof}

\subsubsection{Packing variant}

\begin{proposition}
\label{prop:diffusivity-packing}
Same assumptions as in Proposition \ref{prop:diffusivity}. Assume in addition that the quasisymmetric structure is that of a separable quasimetric space. For every $k,m \in \mathcal{O}^+(u)$, setting $\ell=q \dotplus r \dotplus k$,
\begin{equation}
    \operatorname{pmod}_{p;k}^{\ell,m}(\Gamma) \geqslant \frac{1}{\tau} \int_{\Gamma} d \gamma.
\end{equation}
\end{proposition}

\begin{proof}
Fix $n$, pick a countable $(k \dotplus r, n)$ packing $\mathscr{P}$ of $Z$ with the following condition: for every $\mathbf{a} \in \mathscr{P}$ write $\mathbf a = (a^-, a^+)$, enclosing $(\widehat{\underline b}, \widehat{k.r.\underline b})$ in $\mathbf a$ the $q.r.k.\underline{b}$ cover.  Such packings exist by \ref{lem:covering-lemma}. 
This gives a countable $\mathcal{B} \subset \beta$ (the collection of $\underline b$) such that the realizations of $k.r.\mathcal{B}$ are disjoint. 
Define $\mathcal{Q}_\gamma=\left\{ \underline b \in \mathcal B : \widehat{\ell.\underline{b}} \cap \gamma \neq \emptyset \right\}$. 
The realization of $\ell.\mathcal{Q}_\gamma$ will cover $\gamma$ if $\ell \geqslant q \dotplus r \dotplus k$ and then, by definition of the Carath{\'e}odory measure, $\Phi_{1;0}^{n-\ell(n)}(\gamma) \leqslant \sum_{b\in \mathcal{Q}_\gamma} \widetilde{\phi}^{\ell}(b)$. This gives an inequality equivalent to \eqref{eq:diffusivite-bound-by-sum-gamma-shift} with $\mathcal{Q}_\gamma$ instead of $\mathcal{C}_\gamma$.
The rest of the proof follows the same lines as for Proposition \ref{prop:diffusivity} but instead of \eqref{eq:in-diffusivity-lower-bound-Phi-tilde} one obtains:
\begin{equation}
     \tau_n \mathrm{P}
     {\widetilde \Phi}_{p,(k, n - \ell(n))}^{\ell}(X) \geqslant \tau_n \sum_{b\in \cup_{\gamma} \mathcal{Q}_\gamma} \widetilde{\phi}^\ell(b)^p \geqslant \int_{\Gamma} \Phi_{1,(0, n-\ell(n)}(\gamma)^p d \gamma,
     \label{eq:in-diffusivity-lower-bound-P-Phi-tilde}
\end{equation}
before infimizing over every admissible gauge, which gives a lower bound on $\operatorname{pmod}_{p;k}^{\ell,0}$ and then on $\operatorname{pmod}_{p;k}^{\ell,m}$ for every $m$. 
\end{proof}

\subsection{Conformal dimensions}

\begin{definition}
\label{dfn:conformal-dimension}
Let $(Z,\beta)$ be a $O(u)$-quasisymmetric structure, and let $\Gamma$ be a family of subsets in $Z$.
The $O(u)$-conformal dimension of $\beta$ with respect to $\Gamma$ is 
\begin{align*}
\operatorname{Cdim}^\Gamma_{O(u)}(\beta)  = \sup \Big\{  & p \in \R_{> 0} :  \forall k \in \mathcal{O}^+(u),\, \exists \ell \in \mathcal{O}^+(u)\\
    &  \exists m \in \mathcal{O}^+(u),\, \operatorname{mod}_{p;k}^{\ell,m}(\Gamma, \beta) = +\infty \Big\} 
\end{align*}
or $0$ if this set is empty. Similarly, define
\begin{align*}
\operatorname{PCdim}^{\Gamma}_{O(u)}(\beta)  = \sup \Big\{  & p \in \R_{> 0} :  \forall k \in \mathcal{O}^+(u),\, \exists \ell \in \mathcal{O}^+(u)\\
    &  \exists m \in \mathcal{O}^+(u),\, \operatorname{pmod}_{p;k}^{\ell,m}(\Gamma, \beta) = +\infty \Big\}
\end{align*}
or $0$ if this set is empty.
\end{definition}

\begin{remark}
\label{rem:rewrite-confdim-using-monotonicity}
    Given that moduli decrease with respect to $p$, the conformal dimension $\operatorname{Cdim}^{\Gamma}_{O(u)}(\beta)$ can be {bounded above} by
    \begin{equation}
\inf \Big\{  p \in \R_{> 0} :  \exists k \in \mathcal{O}^+(u),
    \forall \ell,m \in \mathcal{O}^+(u),\, \operatorname{mod}_{p;k}^{\ell,m}(\Gamma, \beta) < + \infty \Big\}  \notag
\end{equation}
or $+\infty$ if this set is empty, and similarly, $\operatorname{PCdim}^\Gamma_{O(u)}(\beta)$ by
\begin{equation}
\inf \Big\{  p \in \R_{> 0} :  \exists k \in \mathcal{O}^+(u),
    \forall \ell,m \in \mathcal{O}^+(u),\, \operatorname{pmod}_{p;k}^{\ell,m}(\Gamma, \beta) < +\infty \Big\}.  \notag
\end{equation}
\end{remark}

\begin{proposition}[Conformal invariance of the conformal dimensions]
\label{prop:Cdim-conf-inv}
Let $\varphi : (Z, \beta) \to (Z', \beta')$ be a $O(u)$-quasisymmetric homeomorphism and let ${\Gamma}$, resp.\ $\Gamma'$ be a family of subsets in $Z$, resp.\ $Z'$, such that $\Gamma' = \lbrace \varphi(\gamma) : \gamma \in \Gamma \rbrace$.  Then  
\begin{align}
    \label{eq:C-dim-conf-inv}
    \operatorname{Cdim}^\Gamma_{O(u)} (\beta) & = \operatorname{Cdim}^{\Gamma'}_{O(u)} (\beta') \\
    \label{eq:PC-dim-conf-inv} \operatorname{PCdim}^{\Gamma}_{O(u)} (\beta) & = \operatorname{PCdim}^{\Gamma'}_{O(u)} (\beta').
\end{align}
\end{proposition}

\begin{proof}
One can assume $Z = Z'$, $\Gamma = \Gamma'$ and that $\varphi$ is the identity map. Let us start with \eqref{eq:C-dim-conf-inv}.
By symmetry we need only prove $\operatorname{Cdim}^\Gamma_{O(u)}(\beta) \leqslant \operatorname{Cdim}^\Gamma_{O(u)}(\beta')$ and  $\operatorname{PCdim}^{\Gamma;N}_{O(u)}(\beta) \leqslant \operatorname{PCdim}^{\Gamma;N}_{O(u)}(\beta')$.
The conformal dimension $\operatorname{Cdim}^\Gamma_{O(u)}(\beta)$ can be rewritten
\begin{align*}
\operatorname{Cdim}^\Gamma_{O(u)}(\beta)  = \sup \Big\{  & p \in \R_{> 0} :  \exists L,M : \mathcal{O}^+(u) \to \mathcal{O}^+(u)  \\
    & \forall k \in \mathcal{O}^+(u),\, \operatorname{mod}_{p;k}^{L(k),M(k)}(\Gamma, \beta) = +\infty \Big\}.
\end{align*}
Now assume that a real number $p$ is in the set defined on the right and let $L$ and $M$ be the corresponding maps from $\mathcal{O}^+(u)$ to itself. 
Define $L' = \eta \circ L \circ \eta'$ and $M' = \eta \circ M \circ \eta'$.
By Lemma \ref{prop:modulus-conf-inv}, for every $k$ and $m$ in $\mathcal{O}^+(u)$,
\begin{align*}
    0 < \operatorname{mod}_{p;\eta'(k)}^{L(\eta'(k)),M(\eta'(k))}(\Gamma, \beta) & \leqslant \operatorname{mod}_{p;k}^{L'(k),M'(k)}(\Gamma, \beta').
\end{align*}
and the left-hand side is infinite, thus $\operatorname{Cdim}_{O(u)}^\Gamma(\beta') > p$, finishing the proof. \eqref{eq:PC-dim-conf-inv} is obtained in the same way.
\end{proof}

In the following, we may omit $\Gamma$ in $\operatorname{Cdim}^{\Gamma}_{O(u)}$ and write $\operatorname{Cdim}_{O(u)}$; this means that $\Gamma$ must be considered the family of nonconstant curves in $Z$.
Note that homeomorphisms preserve nonconstant curves.

\subsection{Upper bound on {$\operatorname{Cdim}_{O(u)}$}}

\begin{lemma}[Conformal dimension is less or equal than Hausdorff dimension]
\label{lem:confdim-less-than-hausdim}
Let $Z$ be a metric space with Hausdorff dimension $q$. Let $\Gamma$ be the family of nonconstant curves in $Z$. Then $\operatorname{Cdim}_{O(u)}^\Gamma Z \leqslant q$.
\end{lemma}

\begin{proof}
In view of remark \ref{rem:rewrite-confdim-using-monotonicity} this will be proved if we can show that for every $\varepsilon \in (0,q)$,
\begin{equation}
    \exists k\in \mathcal{O}^+(u),\, \forall \ell, m \in \mathcal{O}^+(u), \operatorname{mod}_{q+\varepsilon;k}^{\ell,m}(\Gamma)=0.
    \label{eq:to-prove-upperbound-confdim}
\end{equation}
For $s\in(0,1)$ consider $\phi_s \in \mathcal{G}(\beta)$ such that $\phi (\widehat{b}) = e^{-s {\delta}(b)}$ on concrete balls. By comparison with the Hausdorff measures \eqref{eq:comparison-with-Hausdorff-measures-1}, $\Phi_{1;m} \gg \mathcal{H}^1$.
The nonconstant curves have positive $\mathcal{H}^1$ measure by the triangle inequality, so $\phi_s \in \mathcal{G}_{m}(\Gamma)$ for all $s$.
On the other hand, by \eqref{eq:comparison-with-Hausdorff-measures-2}, $(\widetilde \Phi_s)_{q+\varepsilon;k}^\ell \ll \mathcal{H}_{qs+\varepsilon s - \varepsilon'}$ for every small $\varepsilon'$.
For $s$ sufficiently close to $1$ and $\varepsilon'$ sufficiently small, $qs+\varepsilon s- \varepsilon'>q$, so \eqref{eq:to-prove-upperbound-confdim} is attained.
\end{proof}

\section{Applications to large-scale geometry}
\label{sec:large-scale}

Here two metric spaces $Y$ and $Y'$ are said sublinearly biLipschitz equivalent if there exists a sublinearly biLipschitz equivalence $f : Y \to Y'$ (Definition \ref{dfn:SBE}).

\subsection{Heintze groups}
\label{subsec:heintze}

\subsubsection{Definition}
\label{subsubsec:diagonalizable-type}

\begin{definition}
A connected solvable group $S$ is a purely real Heintze group if its Lie algebra sits in a split extension
\begin{equation}
    1 \to \mathfrak{n} \to \mathfrak{s} \to \mathfrak{a} \to 1
\end{equation}
where $\mathfrak{n}$ is the nilradical of $\mathfrak{s}$, $\dim \mathfrak{a}=1$ and the roots associated to $\mathfrak{a} \to \operatorname{Der}(\mathfrak{n})$ are real and positive multiples of each other.
In addition, we say it is of diagonalizable type if $\operatorname{ad}_{\mathfrak{a}}$ is $\mathbf R$-diagonalizable.
\end{definition}

It is convenient to encode a purely real Heintze group type as a pair $(N,\alpha)$ where $N$ is a nilpotent Lie group and $\alpha$ is a derivation of its Lie algebra with real spectrum and lowest eigenvalue $1$, realizing $\mathfrak{a} \to \operatorname{Der}(\mathfrak{n})$ once an infinitesimal generator $\partial_t\in \mathfrak{a}$ has been fixed. 
Such an $\alpha$ being nonsingular, $N$ is the derived subgroup and $(N,\alpha)$ is metabelian if and only if $N$ is abelian. Every Heintze group admits left-invariant negatively curved Riemannian metrics.

The nilradical of a connected solvable group contains an other characteristic subgroup $\operatorname{Exprad}(S)$, defined as the set of exponentially distorted elements (which does not depend on the choice of a left-invariant proper metric) together with $1$. For purely real Heintze groups both are equal\footnote{One reason for this is that $\alpha$ is nonsingular, compare Peng \cite[2.1]{PengCoarseI} keeping in mind that the Cartan subgroup has rank one here.}.

\begin{theorem}[Implied by Cornulier, {\cite[Th 1.2]{CornulierCones11}}]
\label{thm:cornulier}
Let $H$ be a purely real Heintze group with data $(N,\alpha)$. Decompose $\alpha=\sigma+\nu$ where $\sigma$ is semisimple and $\nu$ is a nilpotent derivation of $\mathfrak{n}$ such that $[\sigma,\nu]=0$. Denote by $H_{\infty}$ the purely real Heintze group of diagonalizable type with data $(N,\sigma)$. Then $H$ and $H_{\infty}$ are $O(\log)$-SBE.
\end{theorem}

\subsubsection{Punctured boundary}

From now on, thanks to Theorem \ref{thm:cornulier}  we work with a purely real Heintze group of diagonalizable type $S$ with data $(N,\alpha)$, that is $S=N\rtimes \R$ where, denoting by $t$ the $\R$ coordinate, $t.x=e^{t\alpha}(x)$ for $x \in N$ and we recall that $\alpha$ is diagonalisable with real positive eigenvalues. 
It is known that this eases the computation of conformal dimension: the latter is attained, indeed by an Ahlfors regular metric, whereas for the twisted plane of Example \ref{exm:twisted-plane} it is not \cite[6]{BonkKleinerCdim} (also, one can prove that no distance has this scaling \cite[5.4]{donne2019metric}).

The vertical geodesics with tangent vector $\partial_t$ all end at time $+\infty$ at a distinguished point $\omega$, that we will call the {\em focal point}, and at time $-\infty$ on the punctured boundary $\partial_\infty^\ast S$ so that we can identify the punctured boundary with $N$; through this identification the one-parameter subgroup generated by $\alpha$ is the dilation subgroup of $\partial_\infty^\ast S$.
Note that if $\rho$ and $\rho'$ are any two proper left-invariant continuous real-valued kernels on $\partial_\infty^\ast S$ such that
$\rho(\xi,\eta)=0 \iff \xi=\eta$ and $\rho(e^{t\alpha}\xi,e^{t\alpha}\eta)= e^t\rho( \xi,\eta)$ for all $t,\xi,\eta$ and similarly for $\rho'$, then $\rho$ and $\rho'$ will only differ by multiplicative constants\footnote{This follows from the same compactness argument which proves that all norm topologies on a finite-dimensional vector space are uniformly equivalent.}. There are several ways to construct such kernels; one is the Euclid-Cygan kernel of Paulin and Hersonsky \cite[appendix]{PaulinHersonsky} which depends on a negatively curved metric on $S$. Another one is Hamenst{\"a}dt's \cite[p.456]{HamenstadtBowenMargulis} (see Dymarz-Peng for its use on boundaries of Heintze groups \cite[2]{DymarzPengBiLip}).
Given the formalism developed in \ref{subsec:definitions-sublin-conf} we will rather use $O(1)$-quasisymmetric structures on the punctured boundary of the form below, which may vary according to our needs. 
{
\begin{definition}
\label{dfn:gen-quasisymmetric}
Let $B$ be a compact subset of $N$ containing $1_N$ in its interior.
We say that a $O(1)$-quasisymmetric structure $\beta^\ast$ is generated by $B$ if $\beta^\ast = N \times \Z$ and for all $b = (x,n) \in \beta$ in this product decomposition, $\widehat{b} = xe^{-\alpha n}(B)$ (note that $\widehat{k.b} = xe^{\alpha k}x^{-1} \widehat{b}$).
\end{definition}

We do not fix $B$, nevertheless the resulting structures for $B$, $B'$ are equivalent since one can find $t>0$ such that $e^{-t\alpha}(B') \subseteq B \subseteq e^{\alpha t}(B')$.}
{We denote by $\beta^\ast$ such a structure on $\partial_\infty^\ast S$.

\begin{lemma}
\label{lem:comparison-quasisym}
Let $\Omega$ be a relatively compact subset of $\partial_\infty^\ast S$. 
Let $\beta$ be the quasisymmetric structure on $\partial_\infty S$ associated with a visual kernel with basepoint $o \in S$ (as in Example \ref{exm:space-with-a-quasidist}). 
Then $\beta_{\mid \Omega}$ and $\beta^\ast_{\mid \Omega}$ are equivalent.
\end{lemma}

\begin{proof}
See Figure \ref{fig:comparison-of-quasisymm-structures}. The Euclid-Cygan kernel of $\xi, \eta \in \Omega$ with reference horosphere $\mathcal{H}$ centered at $\omega$ is, up to a bounded additive error (only depending on the hyperbolicity constant), the distance between a geodesic segment $(\xi \eta)$ and the cloud $\top \Omega \subset \mathcal H$ casting its geodesic shadow from $\omega$ over $\Omega$. Now since $\Omega$ has been assumed relatively compact in $\partial_\infty^\ast S$, $\top \Omega$ is bounded, so that by the triangle inequality
\begin{equation*}
    (\xi, \eta)_o = d((\xi \eta), o) + O(1) = d((\xi \eta), \top \Omega) + O(1). 
\end{equation*}
Finally, the Euclid-Cygan kernel induces the structure $\beta^\ast$.
\end{proof}}

\begin{figure}[t]
\centering
\begin{tikzpicture}[line cap=round,line join=round,>=angle 60,x=3.0cm,y=3.0cm]
\clip(-1.8,-1.01) rectangle (2,1.01);
\draw (-0.2,0.881) node[anchor=north west] {$\omega$};
\draw [rotate around={0.:(0.,-0.585)},line width=1.pt,dash pattern=on 1pt off 1pt,color=blue] (0.,-0.585) ellipse (2.346cm and 0.8300cm);
\draw [shift={(-1.489,0.8)},line width=1.pt,color=blue]  plot[domain=-1.019:0.,variable=\t]({1.*1.489*cos(\t r)},{1.*1.4896714077290278*sin(\t r)});
\draw [shift={(1.488,0.8)},line width=1.pt,color=blue]  plot[range=200,domain=3.14159:4.1540,variable=\t]({1.*1.4889*cos(\t r)},{1.*1.488*sin(\t r)});
\draw [shift={(0.,0.)},line width=1.pt,dash pattern=on 2pt off 2pt,color=blue]  plot[range=200,domain=3.9051:5.51966,variable=\t]({1.*1.*cos(\t r)},{1.*1.*sin(\t r)});
\draw [shift={(-0.5117,-0.952)},line width=1.pt]  plot[domain=0.20488:1.786,variable=\t]({1.*0.449*cos(\t r)},{1.*0.4498*sin(\t r)});
\draw [<->,shift={(-3.3374,0.8)},line width=1.pt]  plot[domain=5.8631:5.982,variable=\t]({1.*3.337*cos(\t r)+0.*3.337*sin(\t r)},{1.*3.3374*sin(\t r)});
\draw [shift={(-3.3374,0.8)},line width=1.pt,dash pattern=on 2pt off 2pt]  plot[domain=5.9821:5.99833,variable=\t]({1.*3.337*cos(\t r)+0.*3.337*sin(\t r)},{0.*3.3374*cos(\t r)+1.*3.3374*sin(\t r)});
\draw [shift={(-0.51176,-0.9525)},line width=1.pt,dash pattern=on 2pt off 2pt]  plot[domain=0.03236:0.2048,variable=\t]({1.*0.44987*cos(\t r)},{1.*0.449*sin(\t r)});
\draw [rotate around={0.:(0.,-0.10162)},line width=1.pt,color=blue] (0.,-0.101628) ellipse (0.8862cm and 0.2794cm);
\draw [shift={(0,0.34031808873418073)},line width=1.pt,color=blue]  plot[domain=4.29328:5.133,variable=\t]({1.*0.550*cos(\t r)},{0.*0.5501*cos(\t r)+1.*0.5501*sin(\t r)});
\draw [color=blue](0.3729,-0.62366) node[anchor=north west] {$\Omega$};
\draw [color=blue](-0.1,-0.01459) node[anchor=north west] {$\top \Omega$};
\draw (-0.6770,-0.49413) node[anchor=north west] {$\xi$};
\draw (-0.3,-0.8634) node[anchor=north west] {$\eta$};
\draw [shift={(0.,0.)},line width=1.pt,dash pattern=on 2pt off 2pt]  plot[domain=-0.7635:3.905112,variable=\t]({1.*1.*cos(\t r)},{1.*1.*sin(\t r)});
\draw [shift={(-6.2939E-4,0.340)},line width=1.pt]  plot[domain=-1.14959:4.2932,variable=\t]({1.*0.550624884736279*cos(\t r)},{1.*0.550*sin(\t r)});
\node at (0.9,-0.5) [preaction={fill=white!20}]{$d((\xi\eta), \top\Omega) = - \log \rho(\xi, \eta) + O(1)$};
\begin{scriptsize}
\draw [fill=blue] (0.,0.8) circle (1.5pt);
\draw [fill=black] (-0.607,-0.5130) circle (1.5pt);
\draw [fill=black] (-0.0621,-0.937) circle (1.5pt);
\draw [fill=black] (-0.1344,-0.1378) circle (1.5pt);
\end{scriptsize}
\end{tikzpicture}
\caption[Comparison of quasisymmetric structures]{Quasisymmetric structures on a relatively compact open subset of the punctured boundary.}
\label{fig:comparison-of-quasisymm-structures}
\end{figure}
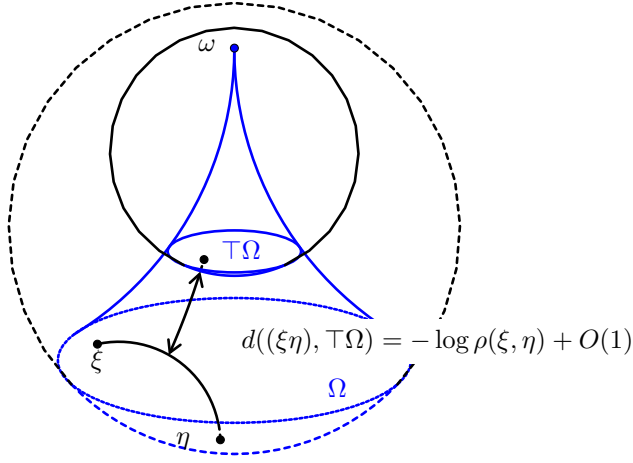

\paragraph{Eigencurves}
For any nonzero eigenvector $v$ of $\alpha$, let $\Gamma_v$ denote the collection of smooth curves in $N$ everywhere tangent to the eigenspace generated by $v$.
A curve $\gamma \in \Gamma_v$ can be parametrized by $\gamma(s)=\gamma(0)e^{sv}$, and thus $\Gamma_v$ is the space of left cosets $N/ \lbrace e^{vt} \rbrace$.
The homogeneous space $\Gamma_v$ has a $N$-invariant, $\alpha$-equivariant measure $\omega_v$ \cite[\S\ 9]{WeilInt}: for any $\lambda$ and nonzero $v\in \ker(\alpha-\lambda)$, for any Borel subset $A$ of $\Gamma_v$,
\begin{equation}
    \label{eq:measure-curve-familie-dilations}
    \omega_v(e^{\alpha t}A)=e^{\operatorname{tr}(\alpha)-\lambda}\omega_v(A).
\end{equation}

\subsubsection[Moduli and conformal dimension]{Moduli of families of eigencurves and conformal dimension}

Let $S$ be a purely real Heintze group of diagonalizable type with data $(N,\alpha)$.
{
If $\Omega$ is an open subset of $\partial_\infty^\ast S$ and $v$ is an eigenvector of $\alpha$, denote by $\Gamma_v(\Omega)$ the set $\lbrace \gamma \cap \Omega : \gamma \in \Gamma_v, \gamma \cap \Omega \neq \emptyset \rbrace$; let us abusively denote $\omega_v$ the measure on $\Gamma_v(\Omega)$.}
The following Lemma corresponds to \cite[2.10 Exemple]{PansuDimConf}.

\begin{lemma}[Lower bound]
\label{lem:lower-bound-heintze}
Let $\lambda \in \mathbf R_{>0}$. 
Let $v \in \ker(\alpha - \lambda)$ be nonzero.
Let $W$ be a $\alpha$-invariant subspace such that $W\oplus \R v =\mathfrak{n}$.
Let {$\beta^\ast_{v,W}$} be the $O(1)$-quasisymmetric structure generated by $B_0=\lbrace \exp P \exp sv \rbrace_{s \in [0,1]}$, where $P\subset W$ is a compact convex subset.
{Let $\Omega \subset \partial_\infty^\ast$ be an open subset and let $\Omega^-$ be an open subset of $\Omega$ such that $\overline{\Omega^-}$ is a concrete ball of $\beta_{v,W}^\ast$.}
For every $\varepsilon > 0$, for every $k\in\mathcal{O}^+(u)$, there exists $\ell\in \mathcal{O}^+(u)$ such that for every $m \in \mathcal{O}^+(u)$,
\begin{equation}
    \label{eq:low-bound-modulus-curves}
    \operatorname{mod}_{(\operatorname{tr}(\alpha)/\lambda)-\varepsilon;k}^{\ell,m} (\Gamma_v(\Omega^-),\beta^\ast_{v,W \mid \Omega})=+\infty,
\end{equation}
and
\begin{equation}
    \label{eq:low-bound-pmodulus-curves}
    \operatorname{pmod}_{(\operatorname{tr}(\alpha)/\lambda)-\varepsilon;k}^{\ell,m}  (\Gamma_v(\Omega^-),\beta^\ast_{v,W \mid \Omega})=+\infty.
\end{equation}
Especially $\operatorname{Cdim}_{O(u)}(\beta^\ast_{\mid \Omega}) \geqslant \operatorname{tr}(\alpha) / \lambda$.
\end{lemma}

\begin{proof}
Set $p = (\operatorname{tr}(\alpha)/\lambda) - \varepsilon$. 
For every $\gamma \in \Gamma_v(\Omega^-)$ let $m_\gamma$ be the Lebesgue measure supported on $\gamma$ with total mass $1$ (the existence is provided by the fact that $\Omega^-$ is relatively compact). 
For every $b \in \beta^\ast_{v,W \mid \Omega}$, letting  $n=\delta(b)$, by \eqref{eq:measure-curve-familie-dilations},
$\omega_v \left\{ \gamma \in \Gamma_v(\Omega^-) : \gamma \cap \widehat{ b} \neq \emptyset \right\} \leqslant \exp \lbrace{-n(\operatorname{tr} \alpha - \lambda)} \rbrace$
while for every $\gamma \in \Gamma_v (\Omega^-)$,
$m_\gamma (\gamma \cap \widehat{1. b})  \geqslant {\mathrm{const}.}  e^{-\lambda n}$ if $\gamma \cap \widehat{{ b}} \neq \emptyset$.
Consequently, 
\begin{align*} 
     \log \int_{\Gamma_v} m_\gamma(\gamma \cap \widehat{{b}})^{1-p} \mathbf{1}_{\gamma \cap \widehat {1.b} \neq \emptyset} d\gamma 
     & \leqslant -n(\operatorname{tr} \alpha - \lambda) - (1-p) \lambda n \\
     & = - \varepsilon \lambda n.
\end{align*}
Thus \eqref{ineq:diffusivity-assumption} is fullfilled for $r =1$ and for every $\tau \in (0,+\infty)$; Propositions \ref{prop:diffusivity} and \ref{prop:diffusivity-packing} then yield \eqref{eq:low-bound-modulus-curves} and \eqref{eq:low-bound-pmodulus-curves} respectively.
{The lower bound on the conformal dimension follows from the definition, viewing $\Gamma_v(\Omega^-)$ as a subcollection of the full collection of nonconstant curves in $\Omega$.}
\end{proof}

\begin{proposition}
\label{prop:computation-Cdim-Heintze}
Let $S$ be a purely real Heintze group of diagonalizable type with data $(N,\alpha)$; assume that the lowest eigenvalue of $\alpha$ is $1$.
{Let $\beta^\ast$ denote a quasisymmetric structure as provided by Definition \ref{dfn:gen-quasisymmetric}}.
Let $\Omega$ be any open subset of $\partial_\infty^\ast S$.
Then
\begin{equation*}
    \label{eq:computation-of-conformal-dimension}
    \mathrm{Cdim}_{O(u)}(\beta^\ast_{\mid \Omega}) = {\operatorname{tr}(\alpha)}.
\end{equation*}
\end{proposition}

\begin{proof}
Under the given assumption that $\alpha$ is diagonalizable, by a theorem of Le Donne and Nicolussi Golo, there exists a true distance on $N$ inducing the quasisymmetric structure $\beta^\ast$ on $\partial_\infty^\ast S$ \cite[Theorem D]{donne2019metric}
, and scaling as $d(e^{t\alpha} \xi, e^{t\alpha} \eta) = e^t d(\xi, \eta)$ while $\operatorname{Jac}(e^t\alpha) = e^{t \cdot \operatorname{tr}(\alpha)}$, so  $\operatorname{Hdim}(d) \leqslant \operatorname{tr}(\alpha)$. By the previous Lemma \ref{lem:lower-bound-heintze} and Lemma \ref{lem:confdim-less-than-hausdim} bounding above conformal dimension with Hausdorff dimension, $\mathrm{Cdim}_{O(u)}(\beta^\ast_{\mid \Omega}) = {\operatorname{tr}(\alpha)}.$
\end{proof}

\begin{lemma}[after {\cite[6D1]{CornulierQIHLC}}]
\label{lem:Cornulier-trick}
Let $S$ and $S'$ be Heintze groups with focal points $\omega$, $\omega'$. If there exists a sublinear biLipschitz equivalence $f_0 : S \to S'$, then there exists a sublinear biLipschitz equivalence $f: S\to S'$ such that $\partial_\infty f(\omega) = \omega'$.
\end{lemma}
We recall that $\partial_\infty f$ extends $f$ to $\partial_\infty S$; its existence is given by Theorem \ref{thm:pallier-SBE2mobius}.

\begin{proof}[Proof]
Let $\widehat f_0: S' \to S$ be a coarse inverse of $f_0$, that is, $d(\widehat f_0 f_0 x, x) = o(d(1_S,x))$ for all $x \in S$ (See \cite[Section 2]{cornulier2017sublinear}). Note that $\partial_\infty (\widehat{f_0}) \partial_\infty f_0= \mathrm{id}_{\partial_\infty S}$. 
    Denote by $\operatorname{SBE}(Y)$ the group of self-sublinear biLipschitz equivalences of $Y$; then $S$ has a homomorphism to $\operatorname{SBE}(S)$ given by $s \mapsto L_s$, where $L_s$ is the left translation by $s$. The image of $S$ acts transitively on $\partial_\infty S \setminus \lbrace \omega \rbrace$, so the action of $\operatorname{SBE}(S)$ on $\partial_\infty S$ has $i$ orbits, with $i \in \lbrace 1,2 \rbrace$. Since $f_0$ conjugates the former action to that of $\operatorname{SBE}(S')$ on $\partial_\infty S'$, the latter also has $i$ orbits.
        If $i=1$ then there is $h \in \operatorname{SBE}(S')$ such that $h(\partial_\infty f_0(\omega)) = \omega'$; set $f = h \circ f_0$.
        Else, if $i=2$, then finite orbits are sent to finite orbits by $f_0$, and $\partial_\infty S' \setminus \lbrace \omega' \rbrace$ being infinite, $f_0(\omega)=\omega'$.
\end{proof}

\begin{proposition}[Generalization of {\cite[Prop 5.9]{pallier2018large}}]
Let $S$ and $S'$ be purely real Heintze groups, write $S = N \rtimes_{\alpha} \R$ and $S' = N \rtimes_{\alpha'} \R$ with normalized $\alpha$ and $\alpha'$. 
If $S$ and $S'$ are sublinearly biLipschitz equivalent then $\operatorname{tr}(\alpha) = \operatorname{tr}(\alpha')$. 
\label{cor:SBE-between-Heintze-spaces}
\end{proposition}

\begin{proof}
By the previously stated theorem \ref{thm:cornulier} of Cornulier we may assume that $S$ and $S'$ are of diagonalizable type. 
Let $\varphi : \partial_\infty^\ast S \to \partial_\infty^\ast S'$ be the boundary mapping of the sublinear biLipschitz equivalence $f$ preserving focal points provided by Lemma \ref{lem:Cornulier-trick}. 
{Let $\Omega$ be a relatively compact subset of $\partial_\infty^\ast S$.}
{Then by Lemma \ref{prop:Cdim-conf-inv}, Theorem \ref{thm:pallier-SBE2mobius} and Lemma \ref{prop:computation-Cdim-Heintze}, letting $\beta^\ast$ and $\beta'^\ast$ be the quasisymmetric structures on $\partial_\infty^\ast S$ and $\partial_\infty S'$ respectively,
\begin{align*}
\operatorname{tr}(\alpha) = \operatorname{Cdim}^\Gamma_{O(u)}(\beta^\ast_{\mid \Omega}) & = \operatorname{Cdim}^{\Gamma'}_{O(u)}(\beta_{\mid \Omega}) \\
& = \operatorname{Cdim}^{\Gamma'}_{O(u)}(\beta'_{\mid \Omega}) = \operatorname{Cdim}^{\Gamma'}_{O(u)}(\beta'^\ast_{\mid \Omega}) =\operatorname{tr}(\alpha'). \qedhere 
\end{align*}
}
\end{proof}

\begin{figure}
    \begin{tikzpicture}[line cap=round,line join=round,>=angle 90,x=0.5cm,y=0.5cm]
        \clip(-5,-5) rectangle (19,5);
        \fill[fill=blue,fill opacity=0.12] (-4,4) -- (-4,-4) -- (4,-4) -- (4,4) -- cycle;
        \fill[fill=blue,fill opacity=0.15] (-1.47,1.05) -- (1.47,1.05) -- (1.47,-1.05) -- (-1.47,-1.05) -- cycle;
        \path[fill=blue,fill opacity=0.2] (-0.54,0.27) -- (0.54,0.27) -- (0.54,-0.27) -- (-0.54,-0.27) -- cycle;
        \draw [->, line width = 0.2pt] (0,-5) -- (0,5);
        \draw [->, line width = 0.2pt] (-5,0) -- (5,0);
        \draw [dash pattern = on 3pt off 2pt] (-4,4) --(4,4) -- (4,-4) -- (-4,-4) -- (-4,4) ;
        \draw [dash pattern = on 3pt off 2pt] (-1.47,1.47) --(1.47,1.47) -- (1.47,-1.47) -- (-1.47,-1.47) -- (-1.47,1.47) ;
        \draw [dash pattern = on 3pt off 2pt] (-0.54,0.54) --(0.54,0.54) -- (0.54,-0.54) -- (-0.54,-0.54) -- (-0.54,0.54)   ;
        \begin{scriptsize}
            \draw (0.6,4.3) node {$1$};
            \draw (0.8,1.75) node {$\small{e^{-1}}$};
            \draw (0.8,0.85) node {$\small{e^{-2}}$};
        \end{scriptsize}
        \draw (-0.2,4) -- (0.2,4);
        \draw (-0.2,1.47) -- (0.2,1.47);
        \draw (-0.2,0.54) -- (0.2,0.54);
        \draw [->, line width = 0.2pt] (12,-5) -- (12,5);
        \draw [->, line width = 0.2pt] (7,0) -- (17,0);
        \fill [fill=blue,fill opacity=0.35] (8,-1) -- (16,-1) -- (16,1) -- (8,1) -- cycle;
        \draw [>-<, line width = 0.2pt] (6.5,1) -- (6.5,-1);
        \draw [>-<, line width = 0.2pt] (8,-2) -- (16,-2);
        \draw (6.5,2)  node {$\sim s$};
        \draw (17.5,-2)  node {$\sim s^{1/\mu}$};
        \draw (14,2) [color =blue, opacity = 0.7] node {$B(s)$};
        \draw (-2.5,-3) [color =blue, opacity = 0.6] node {$B(1)$};
        \draw (11.8,1) -- (12.2,1);
        \draw (12.5,1.4) node{$s$};
    \end{tikzpicture}
    \caption[Boundaries of Heintze groups of diagonalizable type]{Concentric balls of a quasidistance on $\R^2$ that is invariant under translation and dilation by $\exp(t\operatorname{diag}(1, 4/3))_{t \in \R}$, and coincides with the $\ell^\infty$ distance for pairs of points at distance $1$. Compare Figure \ref{fig:example-illustr-thm}.} 
    \label{fig:example-illustr-thm-diag}
\end{figure}
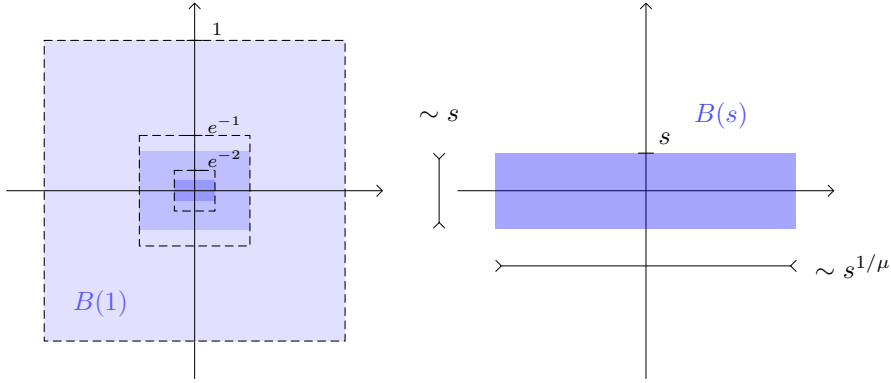

\subsubsection{Proof of the main theorem}
\label{subsubsec:proof-of-main-thm-roots}

Let $u$ be an admissible sublinear function.

\begin{lemma}[Compare {\cite[6.1]{PansuDimConf}} for $u=1$]
\label{lem:lower-bound-energy}
Let $S$ be a Heintze group of diagonalizable type with data $(N,\alpha)$. 
Let $\Omega$ be an open supspace of $\partial_\infty^\ast S$ identified with $N$ and equipped with a quasisymmetric stucture $(\beta, \delta, q)$.
Let $k \in \mathcal{O}^+(u)$.
For every $p \in [1, + \infty)$,  if $f \in \mathscr{V}^{p;k}_{\ell;\mathrm{loc}.}(\Omega)$ with $\ell \geqslant q\dotplus k$ then $f$ is locally invariant along the left cosets of $H$, where 
\begin{equation}
     \mathfrak h = \operatorname{Liespan} \left\{ \operatorname{ker} (\alpha - \mu) : \mu < \frac{ \operatorname{tr}(\alpha)}{p} \right\}.
\end{equation}
\end{lemma}

\begin{proof}

Let $\mu \in (0, \operatorname{tr}(\alpha)/p)$ and let $v \in \ker (\alpha - \mu)$; up to pre-composing $f$ with dilations and translations assume by contradiction that $f(\exp(\varepsilon v)) \neq  f(1)$ for arbitrarily small $\varepsilon$ and that $1 \in \Omega$. 
Up to post-composing $f$ by translations and dilations of $\mathbf R$ one can further assume $f(1)=0$ and $f(\exp(\varepsilon v))>1$.
Construct a condenser $(C, \partial_0 C, \partial_1 C)$ in $\Omega$ as follows: $W$ is a supplementary $\alpha$-invariant subspace of $v$ in $\mathfrak n$, $F$ is a Borel subset of $\exp W$, $C = \lbrace w e^{s  v } : s \in (0, \varepsilon),  w \in F  \rbrace$ and $\partial_i C = \lbrace w e^{i \epsilon  v } \rbrace$.
By Lemma \ref{lem:capacities-and-moduli}, for every $\ell \in \mathcal{O}^+(u)$, $\operatorname{pmod}^{\ell,0}_{p;k}(\Gamma) \leqslant \operatorname{cap}_{p;k}^\ell(C)$, where $\Gamma$ is the family of curves between $\partial_0C$ and $\partial_1C$, which includes $\Gamma_v$. By Lemma \ref{lem:lower-bound-heintze}, $\operatorname{pmod}^{\ell,0}_{p;k}(\Gamma_v)=+\infty$ if $\ell \geqslant q \dotplus k$, and then $\mathbf V_{p;k}^\ell(f)(C)=+\infty$, a contradiction. So $f$ was indeed $\langle v \rangle$-invariant, and then locally invariant on the left cosets of $H$.
Finally, allow $f$ to take complex values. 
\end{proof}

We assume from now on that $N$ is abelian, identify it (as well as $\mathfrak{n}$) with $\R^d$ and decompose
$\mathbf{R}^d = \bigoplus_{i=1}^r \ker (\alpha - \mu_i) = \bigoplus_{i=1}^r \langle e_i^1 \ldots e_i^{d_i} \rangle$.
Let $f^j_i \in (\mathbf R^d)^\vee$ denote the dual basis of linear forms.

\begin{lemma}
\label{lem:upper-bound-energy}
Let $\beta$ be the quasisymmetric structure on $\R^d$ generated by $B=[-1/2,1/2]^d$.
For all $i \in \lbrace 1, \ldots r \rbrace$, for all $j \in \lbrace 1, \ldots d_i \rbrace$, for all $k,\ell \in \mathcal{O}^+(u)$, $f_i^j \in \mathscr{V}^{p;k}_{\ell;\mathrm{loc.}}(\beta, \R)$ for $p > \operatorname{tr}(\alpha)/\mu_i$.
\end{lemma}

\begin{proof}
Let $\nu$ be a Haar measure on $N$, normalized so that $\nu(B) = 1$.
Set $p=(1+\epsilon)\operatorname{tr}(\alpha)/\mu_i$ with $\epsilon>0$.
We need prove that $\mathbf V_{p;k}^\ell(f_i^j)$ is locally finite for every $\epsilon$ and $\ell \in \mathcal{O}^+(u)$. 
We may as well prove that $\mathbf V_{p;k}^\ell(f_i^j)(B)<+\infty$.
Let $n\in \Z_{\geqslant 0}$.
Recall that by definition $\mathbf V_{p;k}^\ell(f_i^j)(B)$ is $\operatorname{P}\widetilde \Phi_{p;k}^{\ell}(B)$ for $\phi(b) = \operatorname{osc}(f_{\mid b})$, so that
$\phi(e^{-\alpha n}B)^p=(e^{-\mu_i n})^p =e^{-\operatorname{tr}(\alpha)(1+\epsilon)n}$ and $\phi$ increases with respect to inclusion. 
If $\mathscr{P}\in \operatorname{Packings}_{k,n}(B)$, enclose into each $(a^-, a^+)$ of $\mathscr{P}$ a pair $(\widehat{\underline b}, \widehat{k.\underline{b}})$ and note that the $\widehat{\underline{b}}$ are disjoint (indeed, even the $\widehat{k. b}$) are); for $n$ large enough they are also contained in $[-1,1]^d$ (since the $\widehat{\underline{b}}$ all intersect $B$) so 
\begin{equation*}
    \sum_{\mathbf{a} \in \mathscr{P}} \nu (\widehat{\underline{b}}) =  \nu \left( \bigcup_{\mathbf a \in \mathscr{P}} \widehat{k.\underline{b}} \right) \leqslant \nu([-1,1]^d) =2^d.
\end{equation*}
From there, and using that $\nu(\widehat{{b}})=e^{-\operatorname{tr}(\alpha)\delta(b)}$ for every $b \in \beta$, and that $\ell$ is sublinear, for $n$ large enough
\begin{equation}
    \sum_{\mathbf{a} \in \mathscr{P}} {\widetilde \phi^\ell}(a^-)^p \leqslant e^{p\ell(n)} \sum_{\mathbf{a} \in \mathscr{P}} \phi(\widehat{\underline{b}})^p \leqslant \sum_{\mathbf{a} \in \mathscr{P}} \nu (\widehat{\underline{b}}) \leqslant 2^d.
\end{equation}
This is a uniform bound for all packings so $\mathbf V_{p;k}^\ell(f_i^j)(B) <+\infty$.
\end{proof}

\begin{remark}
Actually, the $p$-variation of coordinates (or even Lipschitz) functions in the corresponding directions is zero, as can be obtained by replacing $\nu$ with $\mathcal{H}^d$ with $d$ slightly greater than $\operatorname{tr}(\alpha)$ in the previous proof.
To get functions with nonzero yet finite $p$-variation one should form linear combinations of the examples constructed in appendix \ref{exm-non-exm-Ou-qs} composed with coordinates.
\end{remark}

\begin{remark}
The lower bound on variations obtained in the proof of Lemma \ref{lem:lower-bound-energy}, resp.\ the upper bound given by Lemma \ref{lem:upper-bound-energy} can be compared to Xie's \cite[Lemma 4.2]{XieLargeScale} resp.\ \cite[Lemma 4.5]{XieLargeScale}. Xie's technique for the lower bound is essentially different.
\end{remark}

{Let $S$ and $S'$ be two purely real Heintze groups and let $\varphi : \partial_\infty^\ast S \to \partial_\infty^\ast S'$ be the extension of a sublinearly biLipschitz equivalence $f : S \to S'$ preserving the focal points; equip $\partial_\infty^\ast S$ with its abelian Lie group structure and split it into $E_1={\operatorname{span}}_{\mu < \operatorname{tr}(\alpha)/p} \left\{ \ker (\alpha - \mu) \right\}$ and a complementary subspace $E_2$, and similarly decompose $\partial_\infty^\ast S' = E_1' \oplus E_2'$.
For $z \in \partial_\infty^\ast S$, denote by $z_1$ and $z_2$ the projections onto $E_1$ and $E_2$.
Write $\varphi(z_1, z_2) = (\varphi_1 (z_1, z_2), \varphi_2(z_1, z_2))$ where $\varphi_i : E_1 \times E_2 \to E'_i$ for $i \in \lbrace 1, 2 \rbrace$.
For every $(z_1,z_2) \in \partial_\infty^\ast S$, introduce
\[ \mathcal{C}(z) = \left\{ y_1 \in E_1 : \varphi_2(y_1, z_2) = \varphi_2 (z_1, z_2) \right\} \]
and note that $\mathcal{C}(z)$ is nonempty (as it contains $\lbrace{z_1} \rbrace$) and closed.

\begin{lemma}
\label{lem:connectedness-to-prove-that}
For all $z \in \partial_\infty^\ast S$, $\mathcal{C}(z)$ (as defined above) is open in $E_1$.
\end{lemma}

\begin{proof}
As $\mathcal{C}(z) = \mathcal{C}(y_1, z_2)$ for every $y_1 \in \mathcal{C}(z)$, it suffices to prove that $\mathcal{C}(z)$ is a neighborhood of $z_1$.
Let $\Omega$ be a relatively compact open set containing $z$. Denote $\Omega'=\varphi(\Omega)$.
Denote by $\beta$ and $\beta'$ respectively the quasisymmetric structures on $\Omega$ and $\Omega'$ constructed from a Gromov kernel based at $1 \in S, S'$ and denote by $\beta^\ast$ and $\beta'^\ast$  quasisymmetric structures on $\Omega$ and $\Omega'$ associated with Definition \ref{dfn:gen-quasisymmetric}.
Since $\Omega$ and $\Omega'$ have been assumed relatively compact, $\beta$ and $\beta^\ast$ are equivalent by Lemma \ref{lem:comparison-quasisym} and there is a sequence of $O(u)$-quasisymmetric homeomorphisms
\[ (\Omega, \beta^\ast) \overset{\mathrm{id}}{\longrightarrow} (\Omega, \beta) \overset{\varphi}{\longrightarrow} (\Omega', \beta') \overset{\mathrm{id}}{\longrightarrow} (\Omega', \beta'^\ast) \]
Let $\eta, \eta', \overline{\eta}, \overline{\eta}'$ be associated to the $O(u)$-quasisymmetric homeomorphism $\varphi^{-1}:(\Omega', \beta'^\ast) \to (\Omega, \beta^\ast)$ as in \ref{subsec:sublinear-quasiconformality}.
Introduce the following sets: $F = (z+E_1) \cap \Omega$, $F'=(\varphi(z)+E'_1) \cap \Omega'$, and let $F_0$, resp.\ $F'_0$ be the connected component of $F$, resp.\ $F'$ containing $z$, resp.\ $\varphi(z)$.
$F$ is defined inside $\Omega$ by the vanishing of coordinate functions $g_1, \ldots g_s$ with $s = \dim E_2$.
Define $g'_1, \ldots g'_s$ as $g'_i = \varphi_\ast g_i$ ; $\varphi(F)$ is defined in $\Omega'$ by the vanishing of $g'_1, \ldots g'_s$.
Let $q$ be such that axiom (SC2) holds for $\beta^\ast$.
Fix $k, \ell \in \mathcal{O}^+(u)$ such that $\ell \geqslant q \dotplus \overline{\eta}' \circ \eta(k)$.
Using the second embedding in the sequence \eqref{eq:continuous-embeddings-algebras} applied to $\varphi^{-1}$, and the fact that $g_i \in \mathscr{V}^{p;\overline{\eta}(k)}_{\eta'(\ell);\mathrm{loc.}}$ for all $i \in \lbrace{1, \ldots, s}\rbrace$ by Lemma \ref{lem:upper-bound-energy}, one has that $g'_i \in \mathscr{V}^{p;\overline{\eta}'\circ \overline{\eta}(k)}_{\ell;\mathrm{loc.}}(\Omega,\beta^\ast)$ for all $i \in \lbrace{1, \ldots, s}\rbrace$.
By Lemma \ref{lem:lower-bound-energy}, $g'_i$ is locally constant on $F'$, hence zero on its connected component containing $\varphi(z)$. This proves that $\varphi(F_0) \subseteq F'_0$ and the lemma as $F_0$ is open in $z+E_1$.
\end{proof}

By connectedness of $E_1$, Lemma \ref{lem:connectedness-to-prove-that} implies that $\mathcal{C}(z) = E_1$ for all $z$, $\varphi_2$ only depends on the second coordinate $z_2$ and the foliation of $\partial_\infty^\ast S$ by subspaces parallel to $E_1$ is preserved. As $\varphi_2$ is necessarily injective, $s = \dim E_2 \leqslant \dim E_2'$. By symmetry, $\dim E_2 = \dim E_2'$. From there one deduces that
\begin{equation}
     \forall p \in [1, + \infty), \sum_{\mu \geqslant \operatorname{tr}(\alpha)/p} \dim \ker(\alpha-\mu) = \sum_{\mu \geqslant \operatorname{tr}(\alpha')/p} \dim \ker(\alpha'-\mu)
\end{equation}
which implies that $\alpha$ and $\alpha'$ have the same characteristic polynomial. Since they have been assumed diagonalizable with all eigenvalues real and greater or equal than $1$, they are conjugated and the groups $S, S'$ are isomorphic.
}

\subsection{Comparisons and comments}

\subsubsection{$\ell^p$-equivalence relation}
There are other algebras on the boundary of  hyperbolic spaces, the extensions (modulo $\R$) of representatives of $\ell^pH^1(X)$ to $\partial_\infty X$. Bourdon and Kleiner have studied the corresponding equivalence relations, called the $\ell^p$-equivalence relations see e.g. \cite[10]{BourdonKleinerCLPi}.
For Heintze groups of diagonalizable type, comparing our result with that provided by Carrasco Piaggio \cite{CarrascoOrliczHeintze}, the $\ell^p$-equivalence relations coincides with those we obtain for $\mathscr{V}^{p;k}_{\ell;\mathrm{loc.}}$ algebras for adequate $k$ and $\ell$, except perhaps at the critical degrees.

\subsubsection{Quasiisometric classification of diagonalizable Heintze groups}
The result, subsumed by Xie's work \cite{XieLargeScale}, that two quasiisometric purely real metabelian Heintze groups of diagonalisable type are isomorphic is due to Pansu. Sequeira recently recovered it using relative $L^p$-cohomology \cite[Theorem 1.5]{Sequeira}. 
The quasiisometry invariance of the characteristic polynomial of $\alpha$ holds in general
\cite{CarrascoSequeira}.

\subsection{Fuchsian buildings}

\label{subsec:fuchsian}

The point here is to show that $\operatorname{Cdim}_{O(u)}$ equals $\operatorname{Cdim}$ in this case, following Bourdon's proof; we provide a few details of this proof. 

\subsubsection{Fuchsian buildings}

We recall below a definition according to Bourdon \cite[2]{BourdonFuchsII}.
Let $r \geqslant 3$ be an integer, let $R$ be a polygon in $\mathbb H^2$ with $r$ vertices labeled by $\Z/r\Z$ and angles $\pi / m_i$ where $m_i \geqslant 2$ for every $i \in \Z/ r \Z$. $R$ is the fundamental domain for a cocompact Fuchsian representation of the Coxeter group
\[ \mathrm W = \langle s_i \mid s_i^2, (s_{i} s_{i+1})^{m_i} \rangle, \]
where $\langle s_i \rangle$ stabilizes the edge between vertices $i$ and $i+1$. For every $i \in \Z / r \Z$, let $q_i \geqslant 2$ be an integer. Let $\mathbf m, \mathbf q: \Z/r \Z \to \Z_{\geqslant 0}$ be the corresponding data. 
A cell $2$-complex ${\Delta}$ is the geometric realization of a Fuchsian building (we will not distinguish between them) if
\begin{enumerate}[(FB1)]
    \item 
    Each $2$-cell is isomorphic to the labelled $R$, and each $1$-cell with label $i$ lies in exactly $(1+q_i)$ $2$-cells, those are called chambers.
    \item
    Each pair of distinct $2$-chambers is contained in a subcomplex isomorphic (as a labelled cell complex) to the Coxeter complex of $(\mathrm{W}, \lbrace s_i \rbrace)$, those are called apartments.
    \item
    Given two apartments $A$ and $A'$ with at least one common $2$-cell $C$, the identity map of $C$ extends to an isomorphism of labelled complexes $A \to A'$.
\end{enumerate}

The Bourdon buildings are those for which $m = 2$ (they are called right-angled) and $q_i$ are constants. A building of such type always exists provided $p \geqslant 5$, and is uniquely defined\footnote{In general a building of type $(r, \mathbf m, \mathbf q)$ may or may not exist, and may or may not be unique up to isomorphism of labelled complexes.}; it is usually denoted by  $I_{pq}$, where the thickness $q$ designates the constant\footnote{The shift between $q$ and $q_i$ is here to conform with the building of $\operatorname{SL}(3,\Q_\ell)$ where links are incidence graphs of the projective planes over the residue field so that edges are incident to $1+\ell$ cells.} $q_i+1$ and $p$ designates $r$.
Once the chambers are equipped with the hyperbolic metric, Fuchsian building are $\mathrm{CAT}(-1)$ spaces in view of the description of their links and Ballmann's criterion, we refer to \cite{BourdonFuchsII} and reference therein for these facts as well as many examples. 

\subsubsection{Weighted combinatorial distance}
Starting from a Fuchsian building ${\Delta}$ one can associate to it a dual graph $\mathscr{G}({\Delta})$ whose vertices are the chambers of ${\Delta}$, edges record adjacency, and they are assigned length $\log q$ for edges of type $q$.
Choosing any embedding of the Cayley graph of $\operatorname{W}$ with respect to the $\lbrace{s_i} \rbrace$ as a subgraph of $\mathscr{G}(\Delta)$ yields a distance on $\operatorname{W}$; for $w \in \operatorname{W}$, $\vert w \vert_{\mathbf q}$ denotes the length of $w$ for this distance.
The { growth rate} of $\operatorname{W}$ with respect to $\mathbf q$ is 
$\mathscr{T} := \limsup_n \frac{1}{n} \log \sharp \left\{ w \in \operatorname{W} : \vert w \vert_{\mathbf q} \leqslant n \right\}$; this can be made more explicit \cite[3.1.1]{BourdonFuchsII} (for the Bourdon building the growth rate with no weight is $\operatorname{argch}((p-2)/2))$ so that $\mathscr{T}=\operatorname{argch}((p-2)/2)/\log(q-1)$ for $I_{pq}$). 
The distance between two chambers $d,d'$ in $\Delta$ is denoted by $\vert d - d'\vert_{\mathbf q}$, this is $\vert w \vert_{\mathbf q}$ for $w$ such that $d=w.d'$ in any common apartment. The distance $\vert \cdot - \cdot\vert_{\mathbf q}$ on $\mathscr{G}(\Delta)$ is quasiisometric to the $\operatorname{CAT}(-1)$ metric on $\Delta$, especially it is Gromov-hyperbolic.

\subsubsection{Measure on marked apartments}
Given a chamber $c$ in $\Delta$, let $\mathcal{F}_c$ denote the space of embeddings of the Coxeter complex marked at $c$ into $\Delta$.
There is a unique probability measure $\nu$ on $\mathcal{F}_c$ such that for any chamber $d$, $\nu [\pi \in \mathcal{F}_c: \pi \ni d] =e^{-\vert d - c \vert_{\mathbf q}}$ \cite[2.2.4]{BourdonFuchsII}.

\subsubsection{Geodesic metric on the boundary}
The Gromov product on $\partial_\infty \Delta$ associated to $\vert \cdot \vert_{\mathbf q}$ is denoted by $(\xi, \eta) \mapsto \lbrace{\xi,\eta}\rbrace_c$. 
For $\xi,\eta$ in $\partial_\infty \Delta$,
$\varrho(\xi,\eta)=\exp \left(-\mathscr{T}\lbrace{\xi,\eta}\rbrace_c \right)$ and then 
$\delta(\xi,\eta)=\inf \sum \varrho(\xi_i,\xi_{i+1})$
over chains $\xi=\xi_0 \ldots \xi_s=\eta$ in $\partial_\infty \Delta$.
Bourdon proves that $\delta$ and $\varrho$ are comparable (this is the most involved part of the proof; the details for this point are given in \cite[p.362]{BourdonFuchsI}), and that $\operatorname{Hdim}(\partial_\infty \Delta)$ equals $1+1/\mathscr{T}$ \cite[2.2.7]{BourdonFuchsII}.
Once this is proven, $\delta$ induces the same quasisymmetric structure on the boundary, and by Lemma \ref{lem:confdim-less-than-hausdim}, $\operatorname{Cdim}_{O(u)} \partial_\infty \Delta \leqslant 1+1/\mathscr{T}$.

\subsubsection{Diffusivity condition and lower bound}

\begin{lemma}[After Bourdon  {\cite[2.2.2]{BourdonFuchsII}}]
Let $(Z,d)$ be an Ahlfors-regular metric space. 
Let $\beta$ be the associated quasisymmetric structure. Let $\Gamma$ be a family of rectifiable curves in $Z$ whose lengths are  nonzero and bounded above by a uniform constant. Let $d\gamma$ be a measure on $\Gamma$. 
Let $p'$ be greater than $1$.
If there exists $\eta <+\infty$ such that
\begin{equation}
     \forall b \in \beta, \, \log \int_{\Gamma} \mathbf{1}_{\gamma \cap \widehat{b} \neq \emptyset} d\gamma - (1-p') \delta(b) \leqslant \eta,
    \tag{$\mathrm D'(p')$}
    \label{eq:bourdon-diffusivity}
\end{equation}
then $\operatorname{Cdim}_{O(u)}^\Gamma (\beta) \geqslant p'$.
\end{lemma}
Let us check that \eqref{eq:bourdon-diffusivity} implies \eqref{ineq:diffusivity-assumption} provided $p>p'$ and $r \in \mathcal{O}^+(u)$ is nonzero.
Since $\gamma \in \Gamma$ has been assumed rectifiable, they bear  normalized arclength measures $m_\gamma$ of total mass $1$. 
    \begin{figure}[t]
    \begin{center}
    \begin{tikzpicture}[line cap=round,line join=round,>=angle 45,x=1cm,y=1cm]
    \clip(-3,-2.8) rectangle (3,2.8);
    \fill (0,0) [opacity=0.1, color=blue] circle (0.5cm);
    \fill (0,0) [opacity=0.1, color=blue] circle (2.5cm);
    \draw [rotate around={30:(0,0)}] plot[line width = 1pt, domain=-2:3,variable=\t]({\t},{0.1*cos(300*\t)});
    \draw (1,1.2) node{$\gamma$};
    \draw [<->] (0.5,0) -- (2.5,0);
    \draw (1.5,-1) node {$ \left( e^{r(\delta(b))} - 1 \right)e^{-\delta(b)}$} ;
    \draw (-0.5,0.5) [color=blue] node {$ \widehat b$} ;
    \draw (-1.5,1.5) [color=blue] node {$ \widehat r.b$} ;
    \end{tikzpicture}
    \end{center}
\caption[Figure for inequality \eqref{ineq:reverse-triangle}]{Inequality \eqref{ineq:reverse-triangle}.}
\label{fig:reverse-triangle}
\end{figure}
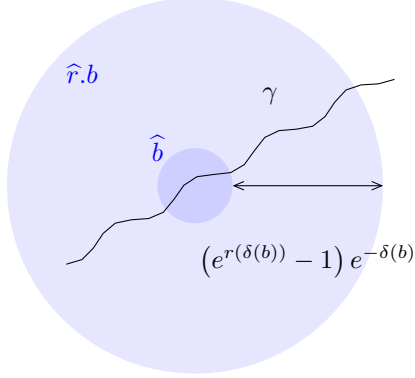
By the reverse triangle inequality, for every $\gamma \in \Gamma$ (see Figure \ref{fig:reverse-triangle}),
\begin{equation}
    \label{ineq:reverse-triangle}
    m_\gamma(\gamma \cap \widehat{r.{b}}) \geqslant 
    \mathbf{1}_{\gamma \nsubseteq \widehat{r.b}}
    \cdot \mathbf{1}_{\gamma \cap \widehat{b} \neq \emptyset} \left( e^{r(\delta(b))} - 1 \right)e^{-\delta(b)}  \operatorname{length}(\gamma)^{-1}  ,
\end{equation}
hence if $\delta(b)$ is large enough to ensure that $\gamma \nsubseteq \widehat{r.b}$, one has:
\begin{align*}
    m_\gamma \left( \gamma \cap  \widehat{r.b} \right)^{1-p} \mathbf{1}_{\gamma \cap \widehat{b} \neq \emptyset} 
    & \leqslant \left( e^{r(\delta(b))} - 1 \right)^{1-p} e^{(p-1)\delta(b)} \mathrm{length}(\gamma)^{p-1}\\
    & \leqslant C\,(1-1/e)^{1-p} \exp \left( (p-1) (\delta(b) -r(\delta(b))) \right)
\end{align*}
where $C = \sup_{\gamma \in \Gamma} \operatorname{length}(\gamma)^{p-1}$ is finite by hypothesis.
Now, using \eqref{eq:bourdon-diffusivity} with $p'<p$,
\begin{equation*}
    \int_{\Gamma} m_\gamma (\gamma \cap \widehat{r.b})^{1-p} \mathbf{1}_{\gamma \cap \widehat{b} \neq \emptyset } d\gamma \leqslant C\,e^\eta e^{( p' -p)\delta(b)+r(\delta(b))}.
\end{equation*}
The right-hand side goes to $0$ because $r$ is sublinear, so \eqref{ineq:diffusivity-assumption} holds for every $\tau \in \R_{>0}$.

Going back to Fuchsian buildings it remains to specify $\Gamma$, $d\gamma$ and $p'$. 
Following Bourdon, given a reference chamber in $\Delta$,
$\Gamma$ is the collection of boundaries of apartments containing the reference chamber $c$ : \[ \Gamma = \left\{ \partial_\infty \mathrm{im}(\pi) : \pi \in \mathcal{F}_c \right\}, \]
$d\gamma$ is the measure on $\Gamma$ corresponding to $\nu$ on $\mathcal{F}_c$.
The fact that the $\gamma \in \Gamma$ are rectifiable follows from \cite[2.2.6(ii)]{BourdonFuchsII}.
The condition \eqref{eq:bourdon-diffusivity} for $p'=1+1/\mathscr{T}$ is checked by Bourdon \cite[2.3.8]{BourdonFuchsII}.
By Lemma \ref{prop:diffusivity}, $\operatorname{Cdim}_{O(u)}(\partial_\infty \Delta) > 1 + 1/ \mathscr{T} -\varepsilon$ for every positive real $\varepsilon$ arbitrarily small. This finishes the proof that $\operatorname{Cdim}_{O(u)} \partial_\infty \Delta = 1+1/\mathscr{T}$. Formula \eqref{eq:confdim-bourdon-buildings} folows for the Bourdon buildings.

\begin{appendix}

\section[Examples and non-properties]{Examples and non-properties of {$O(u)$}-quasisymmetric homeomorphisms}
\label{exm-non-exm-Ou-qs}

We construct and examine here certain $O(u)$-quasisymmetric homeomorphisms of the Euclidean plane. 
The construction uses the observation that products of biLipschitz homeomorphisms are quasisymmetric homeomorphisms. We observe that the homeomorphisms constructed do not possess the ACL property.

The first step of the construction is to build a homeomorphism of the circle with controlled (almost Lipschitz in a precise sense) modulus of continuity.
Let $T$ be a rooted infinite binary tree, whose set of vertices $V$ is identified with the set of finite words over the alphabet $\left\{ 0, 1 \right\}$. 
Let $(\epsilon_j) \in (0,1/2)^\N$ be a decreasing sequence with limit $0$. 
To every $\eta \in \lbrace -1, 1 \rbrace^V$ we associate a homeomorphism $\Phi_\eta$ of the circle as follows:
\begin{enumerate}
    \item 
    for each $v$ of length $\vert v \vert$ one associates a real number $\tau^v$ with the binary expansion $v$ : $\tau^v = \sum_{i=1}^{\vert v \vert} v_i 2^{-i}$.
    \item 
    Let $M_\eta(v)$ be the uniform measure on $[0, 2^{- \vert v \vert}]$ with total mass
    \begin{equation*}
        \label{eq:mass-of-M-eta}
        \Vert M_\eta(v) \Vert = \prod_{w \in \mathrm{Pref}(v) \setminus \lbrace v \rbrace} \left( \frac{1}{2} + \eta(w) \epsilon_{\vert w \vert} \right),
    \end{equation*}
    where $\mathrm{Pref}(v)$ denotes the set of prefixes of $v$ (including the empty one).
    \item 
    For any nonnegative integer $\ell$, 
    $ M_\eta^t := \sum_{v \in V : \vert v \vert = t} \tau_\ast^v M_\eta(v)$, where $\tau_\ast^v$ is the pushforward by the translation $x \mapsto x + \tau^v$.
    \item 
    {Let $\Phi_\eta^t$ be the repartition function of $M_\eta^t$; then $\Phi_\eta^t(\tau_v)$ is constant for $t \geqslant \vert v \vert$, so $\Vert \Phi_\eta^t - \Phi_\eta^{t+1} \Vert_\infty \leqslant \sup _{v : \vert v \vert = t} \Vert M_{\eta}(v) \Vert \leqslant (2/3)^t$ for $t$ large enough. 
    By normal convergence, there exists a uniform limit $\Phi_\eta \in \mathrm{Homeo}^+([0,1])$ of the $\Phi_\eta^t$ as $t \to + \infty$.
    Realizing $S^1$ as $[0,1]/\sim$ where $0 \sim 1$ and considering $\eta$ a random variable one may view $\Phi_\eta$ as a random homeomorphism of the circle.}
\end{enumerate}

\begin{proposition}
\label{prop:not-AC}
If $\epsilon_j \notin \ell^1(\N)$ then $\Phi_\eta$ is not absolutely continuous.
\end{proposition}

\begin{proof}[Proof]
Let $\lambda$ be the Haar measure on $S^1$, and for $t \in \mathbf N_{\geqslant 1}$, let $\Phi^t_\eta$ be the approximation of $\Phi_\eta$ at time $t$ given by $(\Phi^t)' = M^t$. Note that whenever $k$ is an integer with $0 \leqslant k \leqslant 2^t$, one has $\Phi(2^{-t}k) = \Phi^t(2^{-t}k)$.
To every $x \in S^1$ one can associate a geodesic $\gamma_x \subset T$ representing its base $2$ expansion (the finite one for dyadic $x$). Fix $\rho \in (0,1)$. Define
$ A_\eta(\rho) = \left\{ x \in [0,1] : \forall t \in \N, 2^t \Vert M_\eta(\gamma_x(t)) \Vert \geqslant \rho \right\}$.
This is the complementary set in $[0,1]$ of
\begin{align*}
    B_\eta(\rho) & = \left\{ x \in [0,1] : \exists t \in \N, 2^t \Vert M_\eta(\gamma_x(t)) \Vert \leqslant \rho  \right\} \\
    & = \bigcup_{v \in V : \Vert M_\eta(v) \Vert \leqslant 2^{-\vert v\vert} \rho} \left[ \tau^v, \tau^v + 2^{-\vert v \vert} \right] \\
    & = \bigsqcup_{v \in V : 
    \forall w \in \mathrm{Pref}(v) \left( \Vert M_\eta(w) \Vert \leqslant 2^{- \vert w \vert} \rho \implies w = v \right)} \left[ \tau^v, \tau^v + 2^{-\vert v \vert} \right]
\end{align*}
where we used that $\left[ \tau^w, \tau^w + 2^{-\vert w\vert} \right] \supseteq \left[ \tau^v, \tau^v + 2^{-\vert v\vert} \right]$ if and only if $w \in \mathrm{Pref}(v)$, with equality if and only if $v = w$. 
For any $v \in V$ in the set indexing the unions above, $2^{\vert v \vert} \Vert M_\eta(v) \Vert \leqslant \rho$. Now by definition
    \[
    2^{\vert v \vert} \Vert M_\eta(v) \Vert = \frac{\lambda \left( \Phi^{\vert v \vert} [\tau^v, \tau^v + 2 ^{-\vert v \vert}] \right)}{\lambda \left( [\tau^v, \tau^v + 2 ^{-\vert v \vert}] \right) } \]
    so that (omitting the indexation) $\sum_v \lambda \left( \Phi^{\vert v \vert} [\tau^v, \tau^v + 2 ^{-\vert v \vert}] \right) \leqslant \rho \sum_v \lambda \left( [\tau^v, \tau^v + 2 ^{-\vert v \vert} \right)$.
It follows that the $\lambda$-measure of $\Phi_\eta(B(\rho))$ is smaller than $\rho$ for all $\rho$:
\begin{align*}
\lambda(\Phi(B_\eta(\rho))) & =  \sum_{v \in V : 
    \forall w \in \mathrm{Pref}(v) \left( \Vert M_\eta(w) \Vert \leqslant 2^{- \vert w \vert} \rho \implies w = v \right)} \lambda \left( \Phi^{\vert v \vert} \left[ \tau^v, \tau^v + 2^{-\vert v \vert} \right] \right) \\
    & \leqslant \sum_{v \in V : 
    \forall w \in \mathrm{Pref}(v) \left( \Vert M_\eta(w) \Vert \leqslant 2^{- \vert w \vert} \rho \implies w = v \right)} 2^{- \vert v \vert} \rho \leqslant \rho,
\end{align*}
where we used that the intervals $\left[ \tau^v, \tau^v + 2^{-\vert v \vert} \right]$ under consideration are disjoint so that the sum of their measures is $\leqslant 1$.
On the other hand, if $\epsilon_j \notin \ell^1(\N)$ then 
\[ \lambda(B_\eta(\rho)) = 1 - \lambda\left( A_\eta(\rho) \right) = 1 - 0 = 1, \]
since for almost every $x$, the sequence $(2^t \Vert M_\eta (\gamma_x(t)) \Vert)$ is not bounded away from $0$ : up to a null set (the dyadics) one may identify $([0,1], \lambda)$ with the shift space of geodesics rays in $T$ and consider $A_\eta(\rho)$ as an event of probability zero.
Especially $\lambda \left( \bigcap_{\rho \downarrow 0} B_\omega(\rho) \right)  = 1$, whereas the image of this set by $\Phi$ has $\lambda$-measure $0$.
\end{proof}

From now on assume that $\epsilon_j \notin \ell^1(\N)$ but decays sufficiently fast so that the partial sums remain controlled by $u$ :
\begin{equation}
    \label{hypothesis-on-decay-epsilon}
    \sum_{j \leqslant t} \epsilon_j = O(u(t)),
\end{equation}
where we recall that $u$ is strictly sublinear. 
For instance if $\epsilon_j = (3+j)^{-\alpha}$ with $\alpha \in (0,1)$ one may take $u(t)=t^{1-\alpha}$.

\begin{proposition}
\label{prop:modulus-continuity}
Assume that $\epsilon_j$ decays sufficiently fast so that  \eqref{hypothesis-on-decay-epsilon} holds.
Then there exists $v \in O(u)$ such that for all $\eta \in \lbrace 0, 1 \rbrace^V$
\begin{equation}
    \label{eq:expansion-modulus}
    \log l(\Phi_\eta, s) \leqslant \log s + v(-\log s)
\end{equation}
and
\begin{equation}
    \label{eq:modulus-continuity}
    \log s - v(-\log s) \leqslant \log L(\Phi_\eta, s) 
\end{equation}
where 
$l(\Phi_\eta, s) = \sup \left\{ \vert \Phi_\eta(x) - \Phi_\eta(y) \vert : \vert x - y\vert \leqslant s  \right\}$
and  \\
$L(\Phi_\eta, s) = \inf \left\{ \vert \Phi_\eta(x) - \Phi_\eta(y) \vert : \vert x - y\vert \geqslant s  \right\}$.
\end{proposition}

\begin{proof}
Define $t = - \left\lceil \log_2 s \right\rceil$. If $\vert x - y \vert \leqslant s$, then $[x,y]$ is contained in the union of two adjacent dyadic intervals of length $2^{-t}$. Let $\gamma$ and $\gamma'$ be the corresponding geodesic segments in $T$. Then 
\[ \left\vert \Phi_\eta(x) - \Phi_\eta(y) \right\vert \leqslant \Vert M_\eta(\gamma(t)) \Vert + \Vert M_\eta (\gamma'(t)) \Vert \leqslant 2 \prod_{j=0}^{t-1} \left( \frac{1}{2} + \epsilon_j \right), \]
Hence
$
    \log \left\vert \Phi_\eta(x) - \Phi_\eta(y) \right\vert  \leqslant (1-t) \log 2 + \sum_{j=0}^{t-1} \log(1+2 \epsilon_j) 
     \leqslant \log s + v(- \log s)
$
where $v = O(u)$. Similarly, if $\vert x - y \vert \geqslant s$ then $[x,y]$ contains a dyadic interval of length $2^{-1-t}$ with associated geodesic segment $\gamma$ so that 
\[ \left\vert \Phi_\eta(x) - \Phi_\eta(y) \right\vert \geqslant \Vert M_\eta(\gamma) \Vert \geqslant \prod_{j=0}^{t-1} \left( \frac{1}{2} - \epsilon_j \right), \]
providing \eqref{eq:expansion-modulus}.
\end{proof}

\begin{remark}
The aim of Proposition \ref{prop:modulus-continuity} is only to give a modulus of continuity (and a reverse modulus of continuity) for $\Phi_\eta$. 
However we expect the deviation of $\log \left\vert \Phi_\eta(x) - \Phi_\eta(y) \right\vert$ from $\log \vert x - y \vert$ to be typically much lower because of Lindeberg's version of the central limit theorem \cite[Satz II]{LindebergCLT}.
\end{remark}

\begin{remark}
$M_\eta$ is homogeneously multifractal in the sense of Buczolich and Seuret \cite{BuczolichSeuret}, and its multifractal spectrum is concentrated at $\lbrace 1 \rbrace$.
Especially Proposition \ref{prop:modulus-continuity} provides examples for \cite[Proposition 9]{BuczolichSeuret}.
\end{remark}

We can now produce homeomorphisms of $\R$ in the following way: for every $k \in \Z$, choose $\eta_k \in \left\{ -1, 1 \right\}^V$, produce a measure $M_{\eta_k}$ on $[0,1]$, and then set $\mu = \sum_{k \in \Z} k_\ast \mu_{\eta_k}$. Finally $\psi :\R^2 \to \R^2$ is such that $\psi(s) = \int_0^s d \mu$.
This may be considered a random process if $\eta_k$ are considered random variables.

\begin{proposition}
\label{prop:example-SQC}
Let $\Psi : \R^2 \to \R^2$ be defined by $\Psi(x_1,x_2) = (\psi_1(x_1), \psi_2(x_2))$ where $\psi_1$ and $\psi_2$ are as above.
Then $\Psi$ is a $O(u)$-quasisymmetric\footnote{The $O(1)$-quasisymmetric structure, and then the $O(u)$-quasisymmetric structure on $\R^2$, will not depend on the norm, compare \cite[p.78]{HeinonenLectures}.} homeomorphism.
\end{proposition}

\begin{proof}
Equip $\R^2$ with the sup norm.
Rephrasing Definitions \ref{dfn:equivalence} and \ref{dfn:sublin-conf-homeo} we need to prove that for every $K \in \R_{\geqslant 1}$ and $k \in O(u)$ there exists $L \in \R_{\geqslant 1}$ and $\ell \in O(u)$ such that for any sequence $(x^n, y^n, z^n)$ of points in $\R^2$,
\[ \begin{cases}
K^{-1}n \leqslant - \log \Vert y^n -  x^n \Vert \\
- \log \Vert y^n -  x^n \Vert \leqslant K n \\
\left\vert \log \frac{\Vert y^n - x^n \Vert}{\Vert z^n - x^n \Vert} \right\vert \leqslant k(n), \end{cases} \implies
\begin{cases}
L^{-1}n \leqslant- \log \left( \Vert \Psi(y^n) - \Psi(x^n) \Vert \right) \\
- \log \left( \Vert \Psi(y^n) - \Psi(x^n) \Vert \right) \leqslant L n \\
\left\vert \log \frac{\Vert \Psi(y^n) - \Psi(x^n) \Vert}{\Vert \Psi(z^n) - \Psi(x^n) \Vert} \right\vert \leqslant \ell(i).
\end{cases}  \]

Write $x^n = (x^n_1, x^n_2)$, similarly for $y^n$ and $z^n$. 
Let $v \in O(u)$ be such that \eqref{eq:modulus-continuity} holds for every $\psi_\alpha$, i.e.
\begin{equation}
    \label{eq:modulus-applied}
    \forall \alpha \in \lbrace 1, 2 \rbrace, \,  \left\vert \log \vert \psi_\alpha(y) - \psi_\alpha(x) \vert - \log \vert y - x \vert \right\vert \leqslant v (- \log \vert y - x \vert).
\end{equation}
Split $\N$ into three index subsets: 
\[ I^y_1= \left\{ n \in \N : -\log \vert y_2^n - x_2^n \vert > -\log \vert y_1^n -x_1^n \vert + 2 v(n) \right\} \]
\[ I^y_2 = \left\{ n \in \N : -\log \vert y_2^n - x_2^n \vert < -\log \vert y_1^n -x_1^n \vert -2v(n) \right\} \]
\[ I_0^y = \left\{ n \in \N : \left\vert \log \vert y_2^n - x_2^n \vert - \log \vert y_1^n -x_1^n \vert \right\vert \leqslant 2v(n) \right\}. \]
Also, define $I_\alpha^z$ and $J^z_\alpha$ in the same way for $\alpha \in \lbrace 0, 1, 2 \rbrace$. 
Note that since $u$ is non-negative, if $\alpha \neq 0$
\begin{equation}
    \label{eq:norme-sup}
    \forall n \in I_\alpha^y, \, 
    \begin{cases}
        \Vert y^n - x^n \Vert = \vert y_\alpha^n - x_\alpha^n \vert \\
        \Vert \Psi(y^n) - \Psi(x^n) \Vert = \vert \psi_\alpha(y_\alpha^n) - \psi_\alpha(x_\alpha^n) \vert
    \end{cases}
\end{equation}
and similar equalities hold for $n$ in $I_\alpha^z$, whereas if $n \in I_0^y$, resp.\ $n \in I_0^z$ then $\log \Vert y^n - x^n \Vert - \log \vert y_\alpha^n - x_\alpha^n \vert \leqslant 2v(Kn+2v(n))$, resp.\ $\log \Vert z^n - x^n \Vert - \log \vert z_\alpha^n - x_\alpha^n \vert \leqslant 2v(Kn+2v(n))$ for any $\alpha \in \lbrace 1, 2 \rbrace$.
By \eqref{eq:norme-sup}, if $\alpha,\beta \in \lbrace 1,2 \rbrace$ then for $n \in I^y_\alpha \cap I^z_\beta$
\begin{equation*}
   \frac{\Vert y^n - x^n \Vert}{\Vert z^n - x^n \Vert} = \frac{ \vert y_\alpha^n - x_\alpha^n \vert}{ \vert z_\beta^n - x_\beta^n \vert}\; \text{and} \;
     \frac{\Vert \Psi(y^n) - \Psi(x^n) \Vert}{\Vert \Psi(z^n) - \Psi(x^n) \Vert} = \frac{ \vert \psi_\alpha(y_\alpha^n) - \psi_\alpha(x_\alpha^n) \vert}{ \vert \psi_\beta(z_\beta^n) - \psi_\beta(x_\beta^n) \vert}
\end{equation*}
so that, taking logarithms and by \eqref{eq:norme-sup} and \eqref{eq:modulus-applied} and \eqref{eq:norme-sup} again
\begin{align*}
\left\vert \log \frac{\Vert \Psi(y^n) - \Psi(x^n) \Vert}{\Vert \Psi(z^n) - \Psi(x^n) \Vert} : \frac{ \Vert y^n - x^n \Vert}{ \Vert z^n - x^n \Vert} \right\vert & \leqslant
\left\vert \log \frac{\Vert \Psi(y^n) - \Psi(x^n) \Vert}{\Vert \Psi(z^n) - \Psi(x^n) \Vert} - \log \frac{ \vert y_\alpha^n - x_\alpha^n \vert}{ \vert z_\beta^n - x_\beta^n \vert} \right\vert \\
& \leqslant 2v \left( - \inf \lbrace \log \vert y_\alpha^n - x_\alpha^n \vert, \log \vert z_\beta^n - x_\beta^n \vert \rbrace \right) \\
& \leqslant 2v \left( - \inf \left\lbrace \log \Vert y^n - x^n \Vert, \log \Vert z^n - x^n \Vert \right\rbrace \right) \\
& \leqslant 2v(Kn + k(n)).
\end{align*}
It remains to treat the case $n \in I_{\alpha}^y \cap I_\beta^z$ with $\inf \lbrace \alpha, \beta \rbrace = 0$; in this event define $\gamma = \sup \lbrace 1, \alpha, \beta \rbrace$. 
Then
\begin{align*}
\left\vert \log \frac{\Vert \Psi(y^n) - \Psi(x^n) \Vert}{\Vert \Psi(z^n) - \Psi(x^n) \Vert} : \frac{ \Vert y^n - x^n \Vert}{ \Vert z^n - x^n \Vert} \right\vert &  \leqslant
\left\vert \log \frac{\Vert \Psi(y^n) - \Psi(x^n) \Vert}{\Vert \Psi(z^n) - \Psi(x^n) \Vert} - \log \frac{ \vert y_{\gamma}^n - x_{\gamma}^n \vert}{ \vert z_{\gamma}^n - x_{\gamma}^n \vert} \right\vert \\
& \quad + 4v(Kn+2v(n)) \\
& \leqslant 2v(Kn + k(n)) + 2v(Kn+2v(n)).
\end{align*}
Setting $L = K$ and $\ell(n) = k(n) + v(Kn+k(n)) + 4v(Kn+2v(n))$ this finishes the proof. 
\end{proof}

Whereas quasiconformal mappings between open domains of\footnote{Quasisymmetric homeomorphisms of the circle that are not absolutely continuous do exist \cite[IV.B, Remark 2]{AhlforsLectures}.} $\R^2$ have the ACL property (see V{\"a}is{\"a}l{\"a} \cite[32.4]{VaisalaQuasiconformal}; this is instrumental for Mostow rigidity in rank one \cite[\S\ 21]{Mostow70}), Propositions \ref{prop:not-AC} and \ref{prop:example-SQC} imply that it fails for general $O(u)$-quasisymmetric homeomorphisms. 
This is why our main efforts in the article are rather directed to global invariants.

\end{appendix}

\bibliographystyle{amsalpha}

\end{document}